\documentclass[11pt]{article}
\usepackage{graphicx}
\usepackage{amsmath,amsfonts,amsthm}
\usepackage{amssymb,latexsym}
\usepackage{fixmath}
\usepackage{mathrsfs,amsbsy}
\usepackage{dsfont}
\usepackage{enumerate}

\def\wtd{\widetilde}
\def\what{\widehat}

\DeclareMathOperator{\diag}{diag}

\DeclareMathOperator{\subspan}{span}

\DeclareMathOperator{\HH}{H}
\DeclareMathOperator{\T}{T}

\def\bb{\pmb{b}}
\def\bc{\pmb{c}}

\def\be{\pmb{e}}

\def\bh{\pmb{h}}

\def\bk{\pmb{k}}

\def\bn{\pmb{n}}

\def\bq{\pmb{q}}
\def\br{\pmb{r}}
\def\bs{\pmb{s}}

\def\bv{\pmb{v}}
\def\bw{\pmb{w}}
\def\bx{\pmb{x}}
\def\by{\pmb{y}}
\def\bz{\pmb{z}}

\def\sss{\scriptscriptstyle}

\newtheorem{theorem}{Theorem}[section]
\newtheorem{lemma}{Lemma}[section]

\theoremstyle{definition}

\newtheorem{remark}{Remark}[section]
\newtheorem{example}{Example}[section]

\numberwithin{equation}{section}
\numberwithin{figure}{section}
\numberwithin{table}{section}

\allowdisplaybreaks

\usepackage{kbordermatrix}

\usepackage{caption}
\usepackage{amsfonts,amsthm}
\usepackage{amsmath,amssymb,enumerate,color}
\usepackage{algorithm,algorithmic}
\usepackage{epsfig,graphicx,psfrag,mathrsfs,subfigure}
\usepackage{eucal}
\usepackage{yhmath}
\usepackage{hyperref}
\usepackage{diagbox}
\RequirePackage{multirow}
 
\usepackage{textcomp}
\usepackage[]{fontenc}
\usepackage{extarrows}
\def\wtd{\widetilde}
\def\what{\widehat}
\usepackage{rotating}
\usepackage{lscape}
\usepackage{bm}
\usepackage{xspace}
\usepackage{tikz}
\usetikzlibrary{petri} 
\usetikzlibrary{shadows,arrows,positioning,shapes.geometric} 

\usepackage{geometry}
 \usepackage{lipsum}

\def\bb{\pmb{b}}
\def\bc{\pmb{c}}

\def\be{\pmb{e}}

\def\bh{\pmb{h}}

\def\bk{\pmb{k}}

\def\bn{\pmb{n}}

\def\bq{\pmb{q}}
\def\br{\pmb{r}}
\def\bs{\pmb{s}}

\def\bv{\pmb{v}}
\def\bw{\pmb{w}}
\def\bx{\pmb{x}}
\def\by{\pmb{y}}
\def\bz{\pmb{z}}
\def\bzs{\mathbf{0}}

\def\X{\times}

\def\wtd{\widetilde}
\def\what{\widehat}

\def\balp{\boldsymbol{\alpha}}
\def\bphi{\boldsymbol{\varphi}}

\def\bxi{\boldsymbol{\xi}}
\def\bzs{\mathbf{0}}

\def\diag{{\rm diag}}

\def\wtd{\widetilde}
\def\what{\widehat}

\def\bbQ{\mathbb{Q}}
\def\bbP{\mathbb{P}}
\def\bbR{\mathbb{R}}
\def\bbS{\mathbb{S}}

\def\bbV{\mathbb{V}}

\def\bbN{\mathbb{N}}
\def\bbK{\mathbb{K}}

\renewcommand{\algorithmicrequire}{\textbf{Input:}}
\renewcommand{\algorithmicensure}{\textbf{Output:}}

\numberwithin{equation}{section}
\numberwithin{figure}{section}
\numberwithin{table}{section}

\graphicspath{{figures/}{figure/}{pictures/}%
	{picture/}{pic/}{pics/}{image/}{images/}}

\title{Multivariate confluent Vandermonde with $G$-Arnoldi   and applications}
\author{Lei-Hong Zhang\thanks{Corresponding author. School of Mathematical Sciences, Soochow University, Suzhou 215006, Jiangsu, China. This work was
 supported in part by the National Natural Science Foundation of China (NSFC-12471356,  NSFC-12371380), Jiangsu Shuangchuang Project (JSSCTD202209), Academic Degree and Postgraduate Education Reform Project of Jiangsu Province, and China Association of Higher Education under grant 23SX0403. 
        Email: {\tt longzlh@suda.edu.cn}.} \and Ya-Nan Zhang\thanks{School of Mathematical Sciences, Soochow University, Suzhou 215006, Jiangsu, China.  Email: {\tt ynzhang@suda.edu.cn}.} 
\and Linyi Yang\thanks{School of Mathematics and Statistics, Henan University of Science and Technology, Luoyang 471023, Henan, China.  Email: {\tt lyyang@haust.edu.cn}.}\and Yifu Wu\thanks{School of Mathematical Sciences, Soochow University, Suzhou 215006, Jiangsu, China.  Email: {\tt 3121973240@qq.com}. }        }

\date{ }

\begin{document}

\maketitle

\begin{abstract}
In the least-squares fitting framework, the Vandermonde with Arnoldi (V+A) method presented in [Brubeck, Nakatsukasa, and Trefethen,
{\em SIAM Review}, 63 (2021), pp. 405–415] is an effective approach to compute a  polynomial that approximates an underlying univariate function $f$. Extensions of V+A include its multivariate version and the univariate confluent V+A; 
the latter enables us to use the information of the derivative of $f$ in obtaining the approximation polynomial. In this paper, we shall extend V+A further to the multivariate confluent V+A.  
Besides the technical generalization of the univariate confluent V+A, we also introduce a general and application-dependent $G$-orthogonalization in the Arnoldi process. We shall demonstrate with several  applications that, by specifying an application-related $G$-inner product, the desired approximate multivariate polynomial, as well as   certain of its partial derivatives can be computed accurately from a well-conditioned least-squares problem whose coefficient matrix is orthonormal. The desired multivariate polynomial is represented in a discrete   $G$-orthogonal polynomials basis which admits an explicit recurrence, and therefore, facilitates evaluating function values and certain partial derivatives at new nodes efficiently. We demonstrate its flexibility by applying it to solve the multivariate Hermite least-squares problem and PDEs with various boundary conditions in irregular domains.
\end{abstract}


\medskip
{\small
\medskip

}

 
\section{Introduction}\label{sec_intro}
In the least-squares fitting minimization, the Vandermonde with Arnoldi (V+A) method presented by Brubeck, Nakatsukasa and Trefethen,
in \cite{brnt:2021} is an effective approach to compute an approximation polynomial for an underlying univariate function $f$. It consists of the fitting stage and the evaluation stage. Given samples of an underlying function $f$ at a set ${\mathcal X}$ of nodes, in the fitting stage, a discrete orthogonal polynomials basis is generated by the Arnoldi process from the novel monomial basis, and the  fitting approximation polynomial, represented in the discrete orthogonal polynomials basis, can be accurately computed via a well-conditioned least-squares (LS) problem with an orthonormal coefficient matrix. Due to an explicit recurrence of the discrete orthogonal polynomials basis functions, the evaluation stage enables us to evaluate function values  at new nodes efficiently. 
The V+A method turns out to be very effective in many applications (see e.g., \cite{brnt:2021,brtr:2022,hoka:2020,kikr:2014,natr:2021,yazz:2023,zhyy:2025}), and in the presence of rounding error, the forward error analysis is given in \cite{zhsl:2023}.

Besides the application of V+A to various problems, extensions of V+A have also been proposed. For example, the multivariate version of V+A is discussed in \cite{aukl:2021,hoka:2020,zhna:2023} and the univariate confluent V+A  is proposed in \cite{nizz:2023}. The confluent V+A facilitates us to include additional information of the derivative of $f$ to compute the approximation polynomial, and \cite{nizz:2023} demonstrates various applications of the univariate confluent V+A.   

In this paper, we shall present a multivariate confluent V+A which can use both the function values as well as higher-order partial derivatives of $f$ to generate a multivariate approximation polynomial. 
Besides the technical generalization of the univariate confluent V+A, we also introduce a general and application-dependent $G$-orthogonalization in the Arnoldi process \cite{lirc:2004rept}. The general framework of applying this application-dependent $G$-orthogonalization multivariate confluent V+A to various problems will be presented in Section \ref{subsec:Ginner}; in particular,  we shall  demonstrate numerical results from several applications including the multivariate Hermite least-squares problem and solving PDEs with various boundary conditions in irregular domains. These applications imply that, by specifying an application-related $G$-inner product in the Gram-Schmidt (GS) process within the Arnoldi process, the desired approximate multivariate polynomial as well as its certain partial derivatives can be computed accurately from a well-conditioned least-squares problem whose coefficient matrix is orthonormal. Similar to the univariate case, the desired multivariate polynomial is represented in a discrete   $G$-orthogonal polynomials basis which has an explicit recurrence, and therefore, facilitates the user to evaluate function values and certain partial derivatives at new nodes efficiently. 

We organize the paper as follows:  In Section \ref{sec_MVA}, we discuss a particular ordering, namely, the {\it grevlex} ordering,  of the multivariate monomials, and the recurrence relations for the first-order and second-order partial derivatives of the basis functions. Section \ref{sec_MVA} describes the detailed implementations of the multivariate confluent V+A, where the fitting stage and the evaluation stage in the $G$-orthogonalization are presented. Section \ref{sec_apps} presents a general framework to define $G$-orthogonalization in the multivariate confluent V+A and also provides  applications (including the multivariate Hermite least-squares problem and Poisson equations with the Dirichlet condition and/or Neumann condition); through these applications, we demonstrate that, by choosing an application-dependent $G$-inner product in the Arnoldi process, the least-squares problem for the coefficient vector of the approximation multivariate polynomial in the discrete $G$-orthogonal polynomials basis can be well-conditioned, and hence can be solved accurately. No matrix inverse is required and only matrix-vector multiplications involve in the whole process.  Concluding remarks are drawn in Section \ref{sec_conclusion}.

\vskip 2mm 
\noindent{\bf Notation}. 
\begin{itemize}
	\item Throughout the paper, bold lower case letters are used to represent column vectors,  and ${\mathbb R}^{n\times m}$ stands for the set
	of all $n\times m$  real matrices, with the identity $I_n\equiv [\be_1,\be_2,\dots,\be_n]\in\bbR^{n\times n}$, where $\be_i$ is its $i$th column with $i\in [n]:=\{1,2,\dots,n\}$. For a vector $\bx=[x_{1},\dots,x_{n}]^{\T}\in \bbR^n$, we shall also conventionally use $[\bx]_{i}=x_i$ to represent its $i$th entry $x_i$;  
	$\diag(\bx)=\diag(x_{1},\dots,x_{n})$ is the diagonal matrix, and $\|\bx\|_2$ is the $2$-norm  of $\bx$.  
	For  a matrix $A\in\bbR^{m\times n}$,  ${\rm span}(A)$ and $A^{\HH}$ (resp. $A^{\T}$)  are the column space and the conjugate transpose (resp. transpose)  of $A$, respectively. We also adopt MATLAB-like convention to represent the sub-matrix $A_{{\mathcal I}_1,{\mathcal I}_2}$ of $A$, consisting of intersections of rows and columns indexed by ${\mathcal I}_1\subseteq [m]$ and ${\mathcal I}_2\subseteq [n]$, respectively.
	\item 
	Denote by $\bbN^d$ the set of $d$-dimensional vectors consisting of nonnegative integer elements. 
	For the multivariable $\bx=[x_{1},\dots, x_{d}]^{\T}\in \bbR^d$, we denote the multivariate monomial $\bx^{\balp}$ with the multi-index  $\balp=[[\balp]_1,\dots,[\balp]_d]^{\T}\in \bbN^d$  by
	$\bx^{\balp}=x_{1}^{[\balp]_1}x_{2}^{[\balp]_2}\dots x_{d}^{[\balp]_d}$. For a given set of nodes ${\cal X}=\{\bx_j\}_{j=1}^m \subset\bbR^d$, we use boldface  $\bx_j\in \bbR^{d}$ to denote the $j$-th node, distinguishing it from the $j$-th component  $x_j$ in  $\bx$. 
\end{itemize}


\section{Bases of the multivariate polynomials space}\label{sec_Bases} 

Using the multivariate monomial $\bx^{\balp}$, two common multivariate polynomial spaces are the total degree polynomial space and the maximum degree polynomial space defined by
\begin{equation}\nonumber
	\bbP_n^{d,{\rm tol}}:={\rm span}\left(\left\{\prod_{j=1}^d x_j^{[\balp]_j}\right\}_{|\balp|\le n}\right),~~\bbP_n^{d,{\rm max}}:={\rm span}\left(\left\{\prod_{j=1}^d x_j^{[\balp]_j}\right\}_{[\balp]_j\le n}\right),
\end{equation}
respectively. Note that $\dim(\bbP_n^{d,{\rm tol}})=C_{n+d}^d$ and $\dim(\bbP_n^{d,{\rm max}})=(n+1)^d.$
Similar to the univariate case, the multivariate monomials $\{\bx^{\balp}\}$ is a natural basis for the multivariate polynomials space. However, the price of using this convenient basis is its ill-conditioning that   usually causes trouble in numerical computations. For the univariate polynomials, the resulting basis matrix at a given set of nodes is the well-known Vandermonde matrix, whose condition number can usually increase extremely fast with the degree of the polynomial (see e.g., \cite{beck:2000,gaut:1963,high:1990,li:2006a,li:2005e06,li:2008b,pan:2016}). This is the motivation of the Vandermonde with Arnoldi  (V+A) \cite{brnt:2021} where a new basis (a discrete orthogonal polynomials basis) for the polynomial space is constructed so that the basis functions admit an explicit  recurrence and the resulting basis matrix at the given nodes is orthonormal. The  process of constructing the discrete orthogonal polynomials \cite{defp:2010,reic:1993} basis in V+A is also known as the Stieltjes orthogonalization \cite{gaut:2004}, and for the multivariate case, discrete orthogonal polynomials have also been discussed in, e.g., \cite{bach:2009,xuyu:2004}.  The V+A method has proven effective in many applications (see e.g., \cite{brnt:2021,brtr:2022,hoka:2020,kikr:2014,natr:2021,yazz:2023,zhha:2025,zhyy:2025}) and its extensions now include the confluent Vandermonde with Arnoldi \cite{nizz:2023} and the multivariate V+A \cite{aukl:2021,hoka:2020,zhna:2023}. 

To generalize the confluent V+A to the multivariate case, we consider particularly $\bbP_n^{d,{\rm tol}}$ and assume that the  multivariate monomials are in the  {\it grevlex} ordering:
\begin{equation}\label{eq:grevlex}
	\hspace{0.5mm}
	\framebox{\parbox{12cm}{$\bx^{\boldsymbol{\alpha}}=\prod_{i=1}^dx_i^{[\balp]_i}$ appears before $\bx^{\boldsymbol{\wtd\alpha}}=\prod_{i=1}^dx_i^{[\wtd\balp]_i}$ if 
			\begin{itemize}
				\item 
				either $|\boldsymbol{\alpha}|<|\boldsymbol{\wtd\alpha}|$, or 
				\item $|\boldsymbol{\alpha}|=|\boldsymbol{\wtd\alpha}|$ and $\exists k\ge 1$ so that $[\balp]_{j}=[\wtd\balp]_{j}$ for $1\le j\le k-1$ and $[\balp]_{k}>[\wtd\balp]_{k}$.
			\end{itemize}
		}
	}
\end{equation}
As an example of the case $d=2$ and $n=3$, the multivariate polynomials space $\bbP_3^{2,\rm tol}$  can be formed as 
\[
\bbP_3^{2,\rm tol}=\subspan(1,x_1,x_2,x_1^2,x_1x_2,x_2^2,x_1^3,x_1^2x_2,x_1x_2^2,x_2^3).
\]
To simplify our presentation, for $\bbP_n^{d,\rm tol}$, we let $\{\varphi_i(\bx)\}_{i=1}^g$ be the multivariate monomial  basis in the  {\it grevlex} ordering with 
$$g=\dim(\bbP_n^{d,{\rm tol}})=C_{n+d}^d.$$
Based on such a construction, we know that \cite{hoka:2020} {\it for any $\varphi_i(\bx)$ with $i\ge 2$, we can have the smallest index, namely, $s_i<i$, such that $\exists u_i\in [d]$ satisfying}
\begin{equation}\label{eq:generating}
	\varphi_i({\bx}) = x_{u_i}\varphi_{s_{i}}(\bx).
\end{equation}
\begin{lemma}\label{lem:spanin}
	For the  {\it grevlex} ordering basis   $\{\varphi_i\}_{i=1}^g$ of $\bbP_n^{d,\rm tol}$ and any $i\ge 2$, we have 
	\begin{equation}\label{eq:spanin1}
		x_{u_i}\varphi_{j}\in {\rm span}(\varphi_{1},\dots, \varphi_{i-1}), ~~\forall j\le s_i-1,
	\end{equation}
	where   $s_i$ is the smallest index such that $\exists u_i\in [d]$ satisfying \eqref{eq:generating}.
\end{lemma}
\begin{proof} We prove it by contradiction. 
	Assume there is an $j\le s_i-1$ satisfying  $$x_{u_i}\varphi_{j}\not\in {\rm span}(\varphi_{1},\dots, \varphi_{i-1}).$$ Then the smallest $\ell$ satisfying $x_{u_i}\varphi_{j}=\varphi_{\ell}$ must satisfy $\ell\ge i$. Indeed,   we must have $\ell\ge i+1$ because, otherwise, $j$ is the smallest index  (note $j\le s_i-1$) such that $\exists u_i\in [d]$ satisfying \eqref{eq:generating}, a contradiction.
	Let
	\begin{equation*}
		\varphi_{j}(\bx) = \bx^{\boldsymbol{\alpha}_j} = x_1^{[{\balp}_j]_1}\cdots x_d^{[{\balp}_j]_d}~~{\rm and}~~
		\varphi_{s_i}(\bx) = \bx^{\boldsymbol{\alpha}_{s_i}} = x_1^{[{\balp}_{s_i}]_1}\cdots x_d^{[{\balp}_{s_i}]_d}.
	\end{equation*}
	From $\varphi_{i}(\bx)=x_{u_i}\varphi_{s_i}(\bx)$ and $\varphi_{\ell}(\bx)=x_{u_i}\varphi_{j}(\bx)$, it holds that
	\begin{align*}\label{equ:phi}
		\varphi_{i}(\bx) &= \bx^{\boldsymbol{\alpha}_{i}} = x_1^{[\balp_{s_i}]_1}\cdots x_{u_i}^{[\balp_{s_i}]_{u_i}+1}\cdots x_d^{[\balp_{s_i}]_{d}},\\
		\varphi_{\ell}(\bx) &= \bx^{\boldsymbol{\alpha}_{\ell}} = x_1^{[\balp_{j}]_1}\cdots x_{u_i}^{[\balp_{j}]_{u_i}+1}\cdots x_d^{[\balp_j]_{d}}.
	\end{align*}
	On the other hand, as  $j<s_i$, we know that 
	\begin{itemize}
		\item either   $|\boldsymbol{\alpha}_j|<|\boldsymbol{\alpha}_{s_i}|$, or
		\item  $|\boldsymbol{\alpha}_j|=|\boldsymbol{\alpha}_{s_i}|$ and $[\balp_{j}]_k>[\balp_{s_i}]_k$, where $k$ is the smallest index satisfying   $[\balp_j]_{t}= [\balp_{s_i}]_{t}$ for $1\le t\le k-1$, and   $[\balp_{j}]_k> [\balp_{s_i}]_k$. 
	\end{itemize}
	For the former case $|\boldsymbol{\alpha}_j|<|\boldsymbol{\alpha}_{s_i}|$, we have $|\boldsymbol{\alpha}_{\ell}|=|\boldsymbol{\alpha}_{j}|+1<|\boldsymbol{\alpha}_{s_i}|+1=|\boldsymbol{\alpha}_{i}|$, which by the definition in \eqref{eq:grevlex} implies that $\ell <i$ (that is, $\varphi_{\ell}$ appears before $\varphi_{i}$), a contradiction. For the latter case with
	$|\boldsymbol{\alpha}_j|=|\boldsymbol{\alpha}_{s_i}|$ and   $[\balp_{j}]_k>[\balp_{s_i}]_k$,  it is also easy to  see that $|\boldsymbol{\alpha}_i|=|\boldsymbol{\alpha}_{\ell}|$ and the index $k$ is still the smallest one satisfying $[\balp_\ell]_{t}= [\balp_{i}]_{t}$ for $1\le t\le k-1$, and  $[\balp_\ell]_{k}> [\balp_{i}]_{k}$. Again, by the definition in \eqref{eq:grevlex},  it holds  $\ell<i$, which is a contradiction against $\ell\ge i+1$. This proves \eqref{eq:spanin1}.
\end{proof}
For the univariate case, the relation \eqref{eq:generating} reduces to 
\begin{equation}\nonumber
	\varphi_i(x) = x\varphi_{i-1}(x), \quad \mbox{\it i.e., }~ u_i=1 ~{\rm and}~s_i=i-1. 
\end{equation}
We point out that it is such a relation \eqref{eq:generating} and \eqref{eq:spanin1} that make it possible to use the Arnoldi process to construct the new discrete orthogonal polynomials basis. Moreover,  \eqref{eq:generating} also leads to the following formulae for the partial derivatives of $\varphi_i({\bx})$:
\begin{subequations}\label{eq:generating_derivative}
	\begin{align}\label{eq:generating_derivative1}
		\partial_{j} \varphi_i:&=\frac{\partial \varphi_i}{\partial {x_j}}  = \delta_{u_i,j}  \varphi_{s_i}+x_{u_i}\frac{\partial \varphi_{s_i}}{\partial {x_j}},\quad j\in[d],\\ \label{eq:generating_derivative2}
		\partial_{j,k} \varphi_i:&=\frac{\partial^2\varphi_i}{\partial {x_j}\partial {x_k}}  = \delta_{u_i,j}  \frac{\partial \varphi_{s_i}}{\partial {x_k}} +\delta_{u_i,k} \frac{\partial \varphi_{s_i}}{\partial {x_j}}+x_{u_i}\frac{\partial^2\varphi_{s_i}}{\partial {x_j}\partial {x_k}}, \quad   j,k\in[d],
	\end{align}
\end{subequations}
where $\delta_{i,j}=1$ if $i=j$ and $0$ otherwise. Relations \eqref{eq:generating} and  \eqref{eq:generating_derivative}  will be used for generating new discrete orthogonal polynomials basis. 

Besides the basis $\{\varphi_i(\bx)\}_{i=1}^g$ for $\bbP_n^{d,\rm tol}$, the reason we provide \eqref{eq:generating_derivative1} and \eqref{eq:generating_derivative2}  further  is because in many applications, including the multivariate Hermite least-squares problems and solving PDEs, the first-order and the second-order partial derivatives of the underlying function $f:\bbR^d\rightarrow \bbR$ are available and useful in constructing a polynomial approximant for $f$ in a general (irregular) domain $\Omega\subseteq \bbR^d$.  Approximation theorems for the simultaneous approximating a multivariate function and its partial derivatives by a multivariate polynomial and the corresponding partial derivatives can be found in \cite{babl:2002}. In this case, with the basis $\{\varphi_i\}_{i=1}^g$ for $\bbP_n^{d,\rm tol}$, we consider the linear space 
\begin{equation}\label{eq:confPspace}
	{\rm span}(\bbV_n^{(2)})=:{\rm span}(\varphi_1^{(2)},\dots,\varphi_g^{(2)}):={\rm span}\left(\left[\begin{array}{c}\varphi_1 \\\partial_{1}\varphi_1 \\\vdots \\\partial_{d}\varphi_1  \\\partial_{1,1}\varphi_{1}\\\vdots\\\partial_{d,d}\varphi_{1}\end{array}\right],\dots,\left[\begin{array}{c}\varphi_g \\\partial_{1}\varphi_g \\\vdots \\\partial_{d}\varphi_g  \\\partial_{1,1}\varphi_{g}\\\vdots\\\partial_{d,d}\varphi_{g}\end{array}\right]\right).
\end{equation} 
Let $$\wtd d = 1+d+d(d+1)/2.$$ Based on \eqref{eq:generating_derivative}, we can write the recurrence for the basis functions $\{\varphi_i^{(2)}\}_{i=1}^g$ by 
\begin{equation}\label{eq:recurrence0}
	\varphi_i^{(2)}= X_{u_i}^{(2)}\varphi_{s_i}^{(2)},~~i\ge 2,
\end{equation}
where $s_i$ is the smallest index such that $\exists u_i\in [d]$ satisfying \eqref{eq:generating}, and $X_{u_i}^{(2)}\in \bbR^{\wtd d\times \wtd d}$ is a lower-triangular matrix whose diagonal entries are all $x_{u_i}$. Example \ref{eg:d=3} is an illustration for the case $d=3$. Note that for the univariate case $d=1$, the matrix $X_{u_i}^{(2)}$ reduces to  (\cite{nizz:2023})
$$
X^{(2)}=\left[\begin{array}{ccc}x&  &  \\1 & x &  \\ & 2 & x\end{array}\right],
$$
and furthermore, 
$$
{X}^{(\ell)}=\left[\begin{array}{ccccc}x &  &  &  &  \\ 1& x&  &  &  \\ & 2 & x &  &  \\ &  & \ddots &  \ddots&  \\  &   &  &   \ell  & x\end{array}\right]\in \bbR^{(\ell+1)\times  (\ell+1)}.
$$
\begin{example}\label{eg:d=3}
	For $d=3,~u_i=3,~s_i=2$, we have 
	\begin{align}\nonumber
		\varphi_i^{(2)}&=\left[\begin{array}{c}\varphi_i \\\hline\partial_{1}\varphi_i \\ \partial_{2}\varphi_i\\  \partial_{3}\varphi_i  \\\hline\partial_{1,1}\varphi_{i}\\ \partial_{1,2}\varphi_{i}\\ \partial_{1,3}\varphi_{i}\\ \partial_{2,2}\varphi_{i}\\ \partial_{2,3}\varphi_{i}\\ \partial_{3,3}\varphi_{i}\end{array}\right] 
		= \underbrace{\left[\begin{array}{c|ccc|cccccc}x_3 &   &   &   &   &   &   &  &   &   \\\hline   &  x_3 &   &  &   &  &   &  &   &  \\  &   & x_3&  &   &  &   &  &  &  \\1 &   &  &  x_3&   &   &   &  &  &   \\\hline  &   &   &   & x_3 &   &   &   &   &  \\   &   &   &  &   &x_3  &  &   &   &   \\  &1  &   &   &  &  & x_3&  &  &  \\  &   &   &   &   &   &   & x_3 &   &  \\  &    &1   &  &  &   &   &   &  x_3 &   \\  &   &   &  2 &   &   &   &   &  & x_3\end{array}\right]}_{=X_{u_i}^{(2)}}\left[\begin{array}{c}\varphi_2 \\\hline\partial_{1}\varphi_2 \\ \partial_{2}\varphi_2\\  \partial_{3}\varphi_2  \\\hline\partial_{1,1}\varphi_{2}\\ \partial_{1,2}\varphi_{2}\\ \partial_{1,3}\varphi_{2}\\ \partial_{2,2}\varphi_{2}\\ \partial_{2,3}\varphi_{2}\\ \partial_{3,3}\varphi_{2}\end{array}\right]=X_{u_i}^{(2)}\varphi_{s_i}^{(2)}.
	\end{align}
\end{example} 

 Extension of $\bbV_n^{(2)}$ to $\bbV_n^{(3)}$ (or $\bbV_n^{(k)} ~(k>3)$) is not hard as we can stack the third-order ($k$th-order) partial derivatives accordingly; for notational simplicity, we take particularly ${\rm span}(\bbV_n^{(2)})$ for example by assuming that only partial derivatives up to the second-order can be used to compute a polynomial approximant. 

Note that any function 
\begin{equation}\nonumber
	p^{(2)}(\bx)=\sum_{j=1}^g\what c_j\varphi_j^{(2)}(\bx)\in {\rm span}(\bbV_n^{(2)}),~~\what c_j\in \bbR,
\end{equation}
consists of a multivariate polynomial $p(\bx)= \sum_{j=1}^g\what c_j\varphi_j(\bx) \in\bbP_n^{d,\rm tol}$, its $d$ partial derivatives $\{\partial_{j}p\}_{j=1}^d$ and $d(d+1)/2$ second-order partial derivatives $\{\partial_{j,k}p\}_{1\le j\le k\le d}.$ Moreover, when we have a multivariate polynomial $p\in \bbP_n^{d,\rm tol}$ to approximate a target function $f:\bbR^d\rightarrow \bbR$, the corresponding representation w.r.t.  $\bbV_n^{(2)}$ provides $p$, its first-order and second-order partial derivatives too.  Conversely, in case when we have the information of some partial derivatives of the target function $f$ at a given set of nodes (for example, the Hermite least-squares problems and PDEs to be presented in Section \ref{sec_apps}), a least-squares system based on \eqref{eq:confPspace} can also be constructed to compute the approximant $p\in \bbP_n^{d,\rm tol}$.

Given a set ${\mathcal X}=\{\bx_j\}_{j=1}^m$ of nodes, where $\bx_j\in \Omega\subset \bbR^d$ and $\Omega$ is a general (irregular) domain, our process of the multivariate confluent V+A  is a way to construct a new basis, namely, $\{\xi^{(2)}_j\}_{j=1}^g$, for ${\rm span}(\bbV_n^{(2)})$ so that 
\begin{itemize}
	\item [(G1)] it has an explicit recurrence for  $\xi^{(2)}_i$ using $\{\xi^{(2)}_j\}_{j=1}^{i-1}$, and 
	\item[(G2)] a particular coefficient matrix arising from a specific least-squares (LS) application at nodes ${\mathcal X}=\{\bx_j\}_{j=1}^m$ is well-conditioned (orthonormal). The solution of LS yields an approximant $p^{(2)}(\bx)= \sum_{j=1}^gc_j\xi^{(2)}_j(\bx) \in{\rm span}(\bbV_n^{(2)})$ for the associated application.
\end{itemize}
We shall see in Section \ref{sec_MVA} that the property in (G1) facilitates the evaluations of the function value as well as its partial derivatives of the approximant computed accurately from (G2). 
The way to realize these goals is the change of bases. In principle,  any nonsingular $R\in\bbR^{g\times g}$  can produce a new basis $[\varphi_1^{(2)},\dots,\varphi_g^{(2)}]R$ for ${\rm span}(\bbV_n^2)$ (particularly, an upper triangular $R$ ensures that the new basis for $\bbP_n^{d,\rm tol}$ is also degree-graded). However, such a basis may result in an ill-conditioned LS problem and may not fulfill the goal (G2);  moreover, the explicit use of the $[\varphi_1^{(2)},\dots,\varphi_g^{(2)}]$ is numerically unstable as the multivariate monomials $\{\varphi_j\}_{j=1}^g$ are nearly linearly dependent. In the next section, we will describe the Arnoldi process to achieve these goals. 

\section{Multivariate confluent Vandermonde with $G$-Arnoldi}\label{sec_MVA} 
\subsection{The multivariate confluent Vandermonde matrix}\label{subsec_MVnd}
The change of basis  for ${\rm span}(\bbV_n^{(2)})$ uses the information of the given  nodes ${\mathcal X}=\{\bx_j\}_{j=1}^m$ in  a general (irregular) domain $\Omega$. To begin with, the basis  in \eqref{eq:confPspace} for ${\rm span}(\bbV_n^{(2)})$ gives rise to the following multivariate confluent Vandermonde matrix
\begin{equation}\nonumber
	V^{(2)}:=\left[\begin{array}{cccc}\bphi_1 &\bphi_2 &\dots&\bphi_g\\\hline\partial_{1}\bphi_1&\partial_{1}\bphi_2&\dots&\partial_{1}\bphi_g \\\vdots&\vdots&\vdots &\vdots \\\partial_{d}\bphi_1&\partial_{d}\bphi_2&\dots&\partial_{d}\bphi_g  \\\hline\partial_{1,1}\bphi_{1}&\partial_{1,1}\bphi_{2}&\dots&\partial_{1,1}\bphi_{g}\\\vdots&\vdots&\vdots&\vdots\\\partial_{d,d}\bphi_{1}&\partial_{d,d}\bphi_{2}&\dots&\partial_{d,d}\bphi_{g}\end{array}\right] =: \kbordermatrix{ &\sss g \\
		\sss m & V_0   \\ 
		\sss md &  V_1\\
		\sss md(d+1)/2 &  V_2}\in \bbR^{m\wtd d\times g}, 
\end{equation} 
where 
\begin{subequations}\label{eq:phi012}
	\begin{align}\label{eq:phi0}
		\bphi_j& =[\varphi_j(\bx_1),\dots,\varphi_j(\bx_m)]^{\T}\in \bbR^m,\quad 1\le j\le g,  \\\label{eq:phi1}
		\partial_{k}\bphi_j&=[\partial_{k}\varphi_j(\bx_1),\dots,\partial_{k}\varphi_j(\bx_m)]^{\T}\in \bbR^m,\quad 1\le j\le g,~1\le k\le d,\\\label{eq:phi2}
		\partial_{k,i}\bphi_j&=[\partial_{k,i}\varphi_j(\bx_1),\dots,\partial_{k,i}\varphi_j(\bx_m)]^{\T}\in \bbR^m,\quad 1\le j\le g,~1\le k\le i\le d,
	\end{align}
\end{subequations}
and $\wtd d=1+d+d(d+1)/2$. 

For $\bbR^{m\wtd d}$, let us induce it with a positive semidefinite inner product \cite{golr:1983} by a positive semidefinite matrix $G\in \bbR^{m\wtd d\times m\wtd d}$:
\begin{equation}\label{eq:Ginner}
	\langle \bx, \by \rangle_G=\bx^{\T}G\by,
\end{equation}
and the $G$-norm of a vector  $\bx$ is $\|\bx\|_G=\sqrt{\langle \bx, \bx \rangle_G}.$
We shall see in Section \ref{sec_apps} that a particular application can naturally lead to a corresponding $G$ to achieve the goal (G2), i.e., to construct an orthonormal coefficient matrix for the associated least-squares problem.

\subsection{The QR factorization of  $V^{(2)}$}\label{subsec_QR}
Denote the matrix $V^{(2)}$ by
\begin{equation}\nonumber
	V^{(2)}=[\bv_1,\dots,\bv_g].
\end{equation}
To explicitly indicate that each column $\bv_i$ is indeed the evaluation of the function $\varphi_i$ at nodes ${\mathcal X}$, we also write it as 
\begin{equation}\label{eq:vecfun}
	\bv_i={\varphi_i}^{(2)}({\mathcal X}), ~~1\le i\le g.
\end{equation}
Let $\{\xi^{(2)}_i\}_{i=1}^g$ be the desired new basis to be computed for ${\rm span}(\bbV_n^{(2)})$ that fulfills the goals (G1) and (G2). 
Let $V^{(2)}=Q^{(2)}R$ be  the QR factorization with the $G$-inner product \eqref{eq:Ginner} and $Q^{(2)}=[\bq_1,\dotsm\bq_g]\in \bbR^{m\wtd d\times g}$. A relation we shall establish is  
\begin{equation}\label{eq:Qphi}
	Q^{(2)}:=\left[\begin{array}{cccc}\bxi_1 &\bxi_2 &\dots&\bxi_g\\\hline\partial_{1}\bxi_1&\partial_{1}\bxi_2&\dots&\partial_{1}\bxi_g \\\vdots&\vdots&\vdots &\vdots \\\partial_{d}\bxi_1&\partial_{d}\bxi_2&\dots&\partial_{d}\bxi_g  \\\hline\partial_{1,1}\bxi_{1}&\partial_{1,1}\bxi_{2}&\dots&\partial_{1,1}\bxi_{g}\\\vdots&\vdots&\vdots&\vdots\\\partial_{d,d}\bxi_{1}&\partial_{d,d}\bxi_{2}&\dots&\partial_{d,d}\bxi_{g}\end{array}\right] =: \kbordermatrix{ &\sss g \\
		\sss m & Q_0   \\ 
		\sss md &  Q_1\\
		\sss md(d+1)/2 &  Q_2}\in \bbR^{m\wtd d\times g}, 
\end{equation}
where, analogous to \eqref{eq:phi012}, 
\begin{subequations}\nonumber%
	\begin{align}\nonumber
		\bxi_j&= [\xi_j(\bx_1),\dots,\xi_j(\bx_m)]^{\T}\in \bbR^m,\quad 1\le j\le g,  \\\nonumber
		\partial_{k}\bxi_j&=[\partial_{k}\xi_j(\bx_1),\dots,\partial_{k}\xi_j(\bx_m)]^{\T}\in \bbR^m,\quad 1\le j\le g,~1\le k\le d,\\\nonumber
		\partial_{k,i}\bxi_j&=[\partial_{k,i}\xi_j(\bx_1),\dots,\partial_{k,i}\xi_j(\bx_m)]^{\T}\in \bbR^m,\quad 1\le j\le g,~1\le k\le i\le d,
	\end{align}
\end{subequations}
and $Q^{(2)}\be_j=\xi_j^{(2)}({\mathcal X})$. 
The Arnoldi process to be presented is a way to realize \eqref{eq:Qphi} which does not involve explicit computation of the QR factorization of $V^{(2)}$ but also gives the recurrence of $\{\xi^{(2)}_i\}_{i=1}^g$ specified in (G1). 

To describe the process, we note that the QR  of $V^{(2)}=Q^{(2)}R$ is just the Gram-Schmidt process in the $G$-inner product. Particularly, we first normalize the first column $\bv_1$ of $V^{(2)}$ to have 
$
Q^{(2)}\be_1=\bq_1=\frac{\bv_1}{\|\bv_1\|_G},
$
and for  $i\ge 2$, we have 
\begin{equation}\label{eq:GSj}
	\bq_i=\frac{\bv_i-\sum_{k=1}^{i-1}\bq_k\langle \bv_i,\bq_k \rangle_G}{\left\|\bv_i-\sum_{k=1}^{i-1}\bq_k\langle \bv_i,\bq_k \rangle_G\right\|_G}.
\end{equation}
Because the basis functions $\{\varphi_i\}_{i=1}^g$ for $\bbP_n^{d,\rm tol}$ have the relations \eqref{eq:generating} and \eqref{eq:recurrence0}, the evaluation vector $\bv_i$ of $\varphi_{{i}}$ at nodes ${\mathcal X}$  satisfies  
\begin{equation}\label{eq:vectorrecu}
	\bv_i= {\bf X}_{u_i}^{(2)}\bv_{s_{i}},
\end{equation}
where   $s_i$ is the smallest index such that $\exists u_i\in [d]$ satisfying \eqref{eq:generating}, and the matrix ${\bf X}_{u_i}^{(2)}\in \bbR^{\wtd d m\times \wtd d m}$ is generated from the matrix ${X}_{u_i}^{(2)}$ in \eqref{eq:recurrence0} in the Kronecker product fashion by replacing the variable $x_{u_i}$ with a diagonal matrix $$X_{u_i}=\diag([\bx_{1}]_{{u_i}},\dots,[\bx_{m}]_{u_i})\in \bbR^{m\times m}~~([\bx_j]_{u_i}=\bx_j^{\T}\be_{u_i} \mbox{ is the $u_i$th entry of the node $\bx_j$})$$
and multiplying the constants $0,1,2$ by $I_m$. For example,   ${\bf X}_{u_i}^{(2)}$ in Example \ref{eg:d=3} is 
$$
{\bf X}_{u_i}^{(2)}={\left[\begin{array}{c|ccc|cccccc} X_3 &   &   &   &   &   &   &  &   &   \\\hline   &  X_3&   &  &   &  &   &  &   &  \\  &   & X_3&  &   &  &   &  &  &  \\I_m &   &  &  X_3&   &   &   &  &  &   \\\hline  &   &   &   & X_3 &   &   &   &   &  \\   &   &   &  &   &X_3 &  &   &   &   \\  &I_m  &   &   &  &  & X_3&  &  &  \\  &   &   &   &   &   &   & X_3 &   &  \\  &    &I_m   &  &  &   &   &   &  X_3&   \\  &   &   &  2I_m &   &   &   &   &  & X_3\end{array}\right]}, 
$$
$X_3=\diag([\bx_1]_3,\dots,[\bx_m]_3)\in \bbR^{m\times m}$.  For the univariate case $d=1$ with nodes $\{x_j\}_{j=1}^m\subset \bbR$, it is true that with $X=\diag(x_1,\dots,x_m)\in \bbR^{m\times m}$, 
\begin{equation}\label{eq:d1secondX}
	{\bf X}^{(2)}=\left[\begin{array}{ccc}X &  &  \\ I_m& X &  \\ & 2I_m & X\end{array}\right]\in \bbR^{3m\times 3m}, 
\end{equation}
and furthermore \cite{nizz:2023}, 
$$
{\bf X}^{(\ell)}=\left[\begin{array}{ccccc}X &  &  &  &  \\ I_m& X&  &  &  \\ & 2I_m & X &  &  \\ &  & \ddots &  \ddots&  \\  &   &  &   \ell I_m& X\end{array}\right]\in \bbR^{(\ell+1)m\times  (\ell+1)m}.
$$
Using \eqref{eq:vectorrecu}, the GS process of \eqref{eq:GSj} can be equivalently written as
\begin{equation}\label{eq:GSjb}
	\bq_i=\frac{{\bf X}_{u_i}^{(2)}\bv_{s_{i}}-\sum_{k=1}^{i-1}\bq_k\langle {\bf X}_{u_i}^{(2)}\bv_{s_{i}},\bq_k \rangle_G}{\left\|{\bf X}_{u_i}^{(2)}\bv_{s_{i}}-\sum_{k=1}^{i-1}\bq_k\langle {\bf X}_{u_i}^{(2)}\bv_{s_{i}},\bq_k \rangle_G\right\|_G}.
\end{equation}
Assume that the GS process \eqref{eq:GSjb} proceeds $t$ steps without breakdown (i.e., the denominator is zero) and we have $\{\bq_i\}_{i=1}^t$.

\begin{lemma}\label{lem:Xqsiin}
	Let  $\bq_i$ be computed in \eqref{eq:GSjb}. Then for any $1\le i\le t$ before breakdown (i.e., the denominator of \eqref{eq:GSjb}   is zero), we have 
	\begin{equation}\nonumber
		{\bf X}_{u_i}^{(2)}\bq_j\in {\rm span}(\bq_1,\dots,\bq_{i-1})={\rm span}(\bv_1,\dots,\bv_{i-1}),\quad \forall j\le s_i-1,
	\end{equation}
	where $s_i$ is the smallest index such that $\exists u_i\in [d]$ satisfying \eqref{eq:generating}.
\end{lemma}
\begin{proof}
	Denote $V_j^{(2)}=[\bv_1,\dots,\bv_j]$ and  $Q_j^{(2)}=[\bq_1,\dots,\bq_j]$. From \eqref{eq:GSj},  
	we know that the GS process of \eqref{eq:GSjb} implies that $V_j^{(2)}=Q_j^{(2)}R_j$, where $R_j$ is a nonsingular upper triangular matrix. Thus, ${\rm span}(\bq_1,\dots,\bq_{i-1})={\rm span}(\bv_1,\dots,\bv_{i-1})$ follows. 
	
	For ${\bf X}_{u_i}^{(2)}\bq_j\in {\rm span}(\bv_1,\dots,\bv_{i-1})$, because $\bq_t\in {\rm span}(\bv_1,\dots,\bv_{i-1})$ for any $t\le i-1$,  it suffices to show ${\bf X}_{u_i}^{(2)}\bv_j\in {\rm span}(\bv_1,\dots,\bv_{i-1})$ for any $j\le s_i-1$. For this claim, since $\bv_{j}=\varphi_{j}^{(2)}({\mathcal X})$ (see \eqref{eq:vecfun}), it suffices to show $X_{u_i}^{(2)}\varphi_{j}^{(2)}\in {\rm span}(\varphi_{1}^{(2)},\dots, \varphi_{i-1}^{(2)})$ where $X_{u_i}^{(2)}\in \bbR^{\wtd d\times \wtd d}$ is given in \eqref{eq:recurrence0}. By the structure of the functions $\varphi_{i}^{(2)}$ in \eqref{eq:confPspace}, we only need to show that 
	\begin{equation}\nonumber
		x_{u_i}\varphi_{j}\in {\rm span}(\varphi_{1},\dots, \varphi_{i-1}),~~\forall j\le s_i-1,
	\end{equation}
	which is true by Lemma \ref{lem:spanin}. 
\end{proof}

It is noticed that the computation of $\bq_i$ in \eqref{eq:GSjb} requires the explicit use of the column $\bv_{s_i}$ in $V^{(2)}$, which should be computationally avoided due to the ill-conditioning of  $V^{(2)}$. The multivariate V+A is a process for this purpose. In particular, it can be shown that the columns of $\{\bq_i\}_{i=1}^t$ from \eqref{eq:GSjb} are mathematically equivalent to $\{\what \bq_i\}_{i=1}^t$ computed recursively by 
\begin{align}\label{eq:GSjc}
	\what \bq_1=\bq_1; ~~\what\bq_i=\frac{{\bf X}_{u_i}^{(2)}\what\bq_{s_{i}}-\sum_{k=1}^{i-1}\what\bq_k\langle {\bf X}_{u_i}^{(2)}\what\bq_{s_{i}},\what\bq_k \rangle_G}{\left\|{\bf X}_{u_i}^{(2)}\what\bq_{s_{i}}-\sum_{k=1}^{i-1}\what\bq_k\langle {\bf X}_{u_i}^{(2)}\what\bq_{s_{i}},\what\bq_k \rangle_G\right\|_G},~~i\ge 2.
\end{align}

\begin{theorem}\label{thm:GSjc}
	Let $\bq_i$ and $\what \bq_i$ be computed by \eqref{eq:GSjb} and \eqref{eq:GSjc}, respectively. Then we have $\bq_i=\what \bq_i$ for any $1\le i\le t$ before breakdown (i.e., the denominator of \eqref{eq:GSjb}  is zero).
\end{theorem}
\begin{proof}
	We prove the theorem by induction. Let $V_j^{(2)}=[\bv_1,\dots,\bv_j]$,  $Q_j^{(2)}=[\bq_1,\dots,\bq_j]$ and  $\what Q_j^{(2)}=[\what \bq_1,\dots,\what \bq_j]$. Suppose for any $1\le j\le i-1$, we have 
	\begin{equation}\nonumber
		\what Q_j^{(2)}=Q_j^{(2)}, ~V_j^{(2)}=Q_j^{(2)}R_j,
	\end{equation}
	where $R_j\in \bbR^{j\times j}$ is an upper triangular matrix. Let 
	\begin{equation}\label{eq:whatw}
		\bw = {\bf X}_{u_i}^{(2)}\bv_{s_{i}}-\sum_{k=1}^{i-1}\bq_k\langle {\bf X}_{u_i}^{(2)}\bv_{s_{i}},\bq_k \rangle_G~~{\rm and}~~\what \bw={\bf X}_{u_i}^{(2)}\what\bq_{s_{i}}-\sum_{k=1}^{i-1}\what\bq_k\langle {\bf X}_{u_i}^{(2)}\what\bq_{s_{i}},\what\bq_k \rangle_G.
	\end{equation}
	As $\what Q_{i-1}^{(2)}=Q_{i-1}^{(2)}$ and $s_i\le i-1$, we have 
	$$
	\what \bw={\bf X}_{u_i}^{(2)} \bq_{s_{i}}-\sum_{k=1}^{i-1} \bq_k\langle {\bf X}_{u_i}^{(2)} \bq_{s_{i}}, \bq_k \rangle_G.
	$$
	We shall show that $\bw=\alpha\what \bw$ for some $\alpha\in \bbR$. To this end, we use 
	$$\bv_{s_{i}}=V_{s_i}^{(2)}\be_{s_i}=Q_{s_i}^{(2)}R_{s_i}\be_{s_i}=[Q_{s_i-1}^{(2)},\bq_{s_i}]\left[\begin{array}{c}\br \\\alpha\end{array}\right]=Q_{s_i-1}^{(2)}\br+\alpha \bq_{s_i}$$
	to write 
	\begin{equation}\nonumber
		{\bf X}_{u_i}^{(2)}\bv_{s_{i}}={\bf X}_{u_i}^{(2)}Q_{s_i-1}^{(2)}\br+\alpha{\bf X}_{u_i}^{(2)}\bq_{s_i}.
	\end{equation}
	By Lemma \ref{lem:Xqsiin}, one gets that ${\bf X}_{u_i}^{(2)}Q_{s_i-1}^{(2)}\br\in {\rm span}(\bq_1,\dots,\bq_{i-1})$, and we thus can write it as 
	${\bf X}_{u_i}^{(2)}Q_{s_i-1}^{(2)}\br=Q^{(2)}_{i-1}\what \br$ for some $\what \br\in \bbR^{i-1}$. Plugging ${\bf X}_{u_i}^{(2)}\bv_{s_{i}}=Q^{(2)}_{i-1}\what \br+\alpha{\bf X}_{u_i}^{(2)}\bq_{s_i}$ into the formulation of $\bw$ in \eqref{eq:whatw}, we can easily see by $\|\bq_j\|_G=1 ~(j\le i)$ that 
	$$
	\bw=\alpha\left({\bf X}_{u_i}^{(2)} \bq_{s_{i}}-\sum_{k=1}^{i-1} \bq_k\langle {\bf X}_{u_i}^{(2)} \bq_{s_{i}}, \bq_k \rangle_G\right)=\alpha \what \bw.
	$$
	Consequently, 
	$$
	\what\bq_i=\frac{\alpha\bw}{\alpha\|\bw\|_G}=\frac{ \bw}{ \|\bw\|_G}=\bq_i,
	$$ 
	and the proof is completed.
\end{proof}

\subsection{The multivariate confluent Vandermonde with $G$-Arnoldi}\label{subsec:MV+A}
The new implementation given in \eqref{eq:GSjc} leads to the fitting stage of the multivariate confluent Vandermonde with $G$-Arnoldi (MV+G-A). In particular, by Theorem \ref{thm:GSjc}, related to $\{\bq_j\}_{j=1}^t$, if we define a matrix $K_t^{(2)}=[\bk_1,\dots,\bk_t]$ with 
\begin{equation}\label{eq:Kryvec}
	\bk_1 = \bq_1, ~~\bk_i = {\bf X}_{u_i}^{(2)} \bq_{s_i},~~i\ge 2,
\end{equation}
where $s_i$ is the smallest index such that $\exists u_i\in [d]$ satisfying \eqref{eq:generating}, then  \eqref{eq:GSjc}  corresponds to the QR decomposition of $K_t^{(2)}$:
\begin{equation}\label{eq:KQR}
	K_t^{(2)}=Q_t^{(2)}\wtd R,
\end{equation}
where 
$$
\wtd R_{k,i} = \langle \bk_i, \bq_k \rangle_G~~ (k\le i-1),~~
\wtd R_{i,i} = \left\|\bk_i-\sum_{k=1}^{i-1} \bq_k \wtd R_{k,i} \right\|_G,~~1\le i\le t.
$$
The details are presented in Algorithm \ref{alg:MCV+A-F}, which is a generalization of the multivariate V+A in \cite[Algorithm 2.1]{hoka:2020}.

\begin{algorithm}[thb]
\caption{MV+G-A(F): Fitting stage of the Multivariate Confluent Vandermonde with $G$-Arnoldi} \label{alg:MCV+A-F}
\begin{algorithmic}[1]
\renewcommand{\algorithmicrequire}{\textbf{Input:}}
\renewcommand{\algorithmicensure}{\textbf{Output:}}
\REQUIRE ${\cal X}=\{\bx_j\}_{j=1}^m\subset\bbR^d$, and the {\it grevlex} ordering basis   $\{\varphi_i\}_{i=1}^g$ of $\bbP_n^{d,\rm tol}$.
\ENSURE $Q^{(2)}\in\bbR^{m\wtd d\X t},\wtd{R}\in\bbR^{t\X t}$. 
\STATE $Q^{(2)}\leftarrow 0_{m\wtd d\X g},\wtd{R}\leftarrow 0_{g\X g},~t\leftarrow g$
\STATE $[Q]_{:,1}\leftarrow \be=[1,\dots,1]^{\T}\in\bbR^{m\wtd d}$
\STATE $[\wtd{R}]_{1,1}=\|\be\|_G$
\STATE $[Q^{(2)}]_{:,1}\leftarrow [Q^{(2)}]_{:,1}/[\wtd{R}]_{1,1}$
\FOR{$i=2:g$}
\STATE pick the smallest $s_i\in[d]$ such that $\exists u_i\in[d]$ satisfying  $\varphi_i = x_{u_i}\varphi_{s_i}$
\STATE $\bq_{i} \leftarrow  {\bf X}_{u_i}^{(2)}\bq_{s_i}$
\FOR{$r=1,2$}
\STATE $\bs\leftarrow \langle [Q^{(2)}]_{:,1:{i-1}}, \bq_{i}\rangle_G$
\STATE $\bq_{i} \leftarrow \bq_{i}-[Q^{(2)}]_{:,1:i-1}\bs$
\STATE $[\wtd{R}]_{1:i-1,i}\leftarrow [\wtd{R}]_{1:i-1,i}+\bs$
\ENDFOR
\IF{$\|\bq_{i}\|_G=0$}
\STATE $t\leftarrow i-1$ and {\bf breakdown}
\ELSE
\STATE $[\wtd{R}]_{i,i}\leftarrow \|\bq_{i}\|_G$
\STATE $[Q^{(2)}]_{:,i}\leftarrow \bq_{i}/[\wtd{R}]_{i,i}$
\ENDIF
\ENDFOR
\end{algorithmic}
\end{algorithm}

\begin{remark}
	\begin{enumerate}
	\item In Algorithm \ref{alg:MCV+A-F}, we have used   MATLAB-like notation to represent, for example, $Q_{i-1}^{(2)}$ by $[Q^{(2)}]_{:,1:{i-1}}$. Moreover, $\langle [Q^{(2)}]_{:,1:{i-1}}, \bq_{i}\rangle_G$ in step 9 means 
	$$\langle [Q^{(2)}]_{:,1:{i-1}}, \bq_{i}\rangle_G=[\langle \bq_1, \bq_i \rangle_G, \dots, \langle \bq_{i-1}, \bq_i \rangle_G]^{\T}\in \bbR^{i-1}.$$
	\item In practice, we do not need  the explicit form of $ {\bf X}_{u_i}^{(2)}$ in step 7; instead, one should use the relations given in \eqref{eq:generating_derivative}  to compute ${\bf X}_{u_i}^{(2)}\bq_{s_i}$.
	
	\item The loop between steps 8 and 12 uses the re-orthogonalization of GS process. 
	
	\item
	When it comes to the univariate case $d=1$, \eqref{eq:KQR}  reduces to the (second-order derivative) confluent Vandermonde with Arnoldi \cite{nizz:2023}, where ${\bf X}_{u_i}^{(2)}={\bf X}^{(2)}$ is defined by \eqref{eq:d1secondX} and $s_i=i-1$; also when only the first-order derivatives are used, then \eqref{eq:KQR} becomes $K_t^{(1)}=Q_t^{(1)}\wtd R$ which is the confluent Vandermonde with Arnoldi \cite{nizz:2023}; similarly, $K_t^{(0)}=Q_t^{(0)}\wtd R$ is the original V+A \cite{brnt:2021}. For this case with $d=1$ and each $\ell\ge 0$, we have (see e.g., \cite[Theorem 2.1]{zhsl:2023})
	$$
	[\bq_1, {{\bf X}^{(\ell)}\bq_1,\left({\bf X}^{(\ell)}\right)^2\bq_1,\dots,\left({\bf X}^{(\ell)}\right)^{j-1}\bq_1}]=Q_j^{(\ell)}\Pi_j,~~1\le j\le t-1,
	$$
	where $\Pi_j\in \bbR^{j\times j}$ is an upper triangular matrix. The confluent Vandermonde with Arnoldi \cite{nizz:2023}
	provides    (see e.g., \cite[Theorem 2.1]{zhsl:2023}) the following compact matrix form
	\begin{equation}\label{eq:Arnoldid=1}
		{\bf X}^{(\ell)}Q_j^{(\ell)}=Q_{j}^{(\ell)}H_j+\gamma_j \bq_{j+1}\be_j^{\T},
	\end{equation}
	where $H_j$ is the upper Hessenberg matrix arising from the Arnoldi process. Indeed, from \eqref{eq:KQR}, we know that $H_j=\wtd R_{1:j,2:j+1}$ and $\gamma_{j}=\wtd R_{j+1,j+1}$. The relation \cite[Theorem 2.1]{zhsl:2023} between $H_j$ and $\Pi_j$ is 
	$$\Pi_j=[\be_1,H_j\be_1,\dots,H_j^{j-1}\be_1]\in \bbR^{j\times j}.$$ 
	The (confluent) Vandermonde with Arnoldi \cite{brnt:2021,nizz:2023} is a process for \eqref{eq:Arnoldid=1}.
	\item 
	The name of $G$-Arnoldi process in Algorithm \ref{alg:MCV+A-F} is borrowed from \cite{lirc:2004rept}, where the $G$-inner product is introduced for the structural preserving model reductions and the corresponding $G$-Lanczos process is proposed. In \cite{lirc:2004rept}, the matrix $G$ can be indefinite and the associated $G$-inner product is indefinite \cite{golr:1983} too. In our case, $G$ is positive semidefinite and each particular application could lead to its individual $G$ which yields a well-conditioned LS system for computing the coefficient vector $\bc$ in the multivariate polynomial represented in the new basis functions $\{\xi_j^{(2)}\}_{j=1}^g$ for ${\rm span}(\bbV_n^{(2)})$.
	\item For a positive semidefinite matrix $G$, it is possible that the breakdown with $t<g$ occurs. 
	This is the same as the {\it serious breakdown} in the Krylov subspace methods \cite[p. 389]{wilk:1988}.  For the case when $G$ is a diagonal matrix with diagonal entries either 0 or 1 (the associated $G$-Arnoldi process is called a {\it sub-orthogonalized Arnoldi} process), treatments \cite[Section 9]{lirc:2004rept} for the serious breakdown have been discussed.
	
	\item For the computational complexity, by noting that ${\bf X}_{u_i}^{(2)}\bq_{s_i}$ in step 7 of Algorithm \ref{alg:MCV+A-F}  involves only $O(m\wtd d)$, it requires $O(m\wtd d g^2)$ flops.
	\end{enumerate}
\end{remark}

 \subsection{The recurrence for $\{\xi^{(2)}_i\}_{i=1}^g$ and the evaluation stage}\label{subsec_GArnoldi}
Assume $t=g$. In terms of basis functions for ${\rm span}(\bbV_n^{(2)})$, the nonsingular upper triangular matrix $ R$ in $V^{(2)}=Q^{(2)}R$ provides a new basis for ${\rm span}(\bbV_n^{(2)})$, i.e.,
\begin{equation}\nonumber
	\bbQ_n^{(2)}=\bbV_n^{(2)}R^{-1}=:[\xi_1^{(2)},\dots,\xi_g^{(2)}],
\end{equation}
and thus
\begin{equation}\nonumber
	[\xi_1^{(2)}({\mathcal X}),\dots,\xi_g^{(2)}({\mathcal X})]= [\varphi_1^{(2)}({\mathcal X}),\dots,\varphi_g^{(2)}({\mathcal X})]R^{-1}=[\bq_1,\dots,\bq_g]\in \bbR^{m\wtd d\X g},
\end{equation}
where $\varphi_i^{(2)}({\mathcal X})$ (and similarly $\xi_i^{(2)}({\mathcal X})$) for $1\le i\le g$ are defined in \eqref{eq:vecfun}. Note that $\{\xi_j^{(2)}\}_{j=1}^g$ is a  discrete (on ${\mathcal X}$) orthogonal polynomials basis w.r.t.  the $G$-inner product. From the new basis  $\bbQ_n^{(2)}$,  relations \eqref{eq:Kryvec} and \eqref{eq:KQR} also generate  a basis  for ${\rm span}(\bbV_n^{(2)})$:
\begin{equation}\label{eq:basisfK}
	\bbK_n^{(2)}=\bbQ_n^{(2)}\wtd R=:[\kappa_1^{(2)},\dots,\kappa_g^{(2)}]
\end{equation}
satisfying 
$$\kappa_1^{(2)}({\mathcal X})=\bk_1=\bq_1,~\kappa_i^{(2)}({\mathcal X})=\bk_i= {\bf X}_{u_i}^{(2)} \xi^{(2)}_{s_i}({\mathcal X}),~~i\ge 2.$$
The above relation yields a recurrence for $\{\kappa_j^{(2)}\}_{j=1}^g$:
\begin{equation}\nonumber
	\kappa_1^{(2)}(\bx)=\xi_1^{(2)}(\bx),~~ \kappa_i^{(2)}(\bx)=X_{u_i}\xi_{s_i}^{(2)}(\bx),~~i\ge 2,
\end{equation}
where $s_i$ is the smallest index such that $\exists u_i\in [d]$ satisfying \eqref{eq:generating}. Most importantly, the relation  \eqref{eq:GSjc} provides a recurrence for $\{\xi_j^{(2)}\}_{j=1}^g$:
\begin{equation}\label{eq:recuxi}
	\wtd R_{i,i} \xi_i^{(2)}(\bx) =X_{u_i}\xi_{s_i}^{(2)}(\bx)-\sum_{k=1}^{i-1}\wtd R_{k,i}\xi_k^{(2)}(\bx),
\end{equation}
where the upper triangular matrix $\wtd R$ is computed from Algorithm \ref{alg:MCV+A-F}. 
It is this recurrence for the discrete (on ${\mathcal X}$) orthogonal polynomials $\{\xi_j^{(2)}\}_{j=1}^g$ that makes the evaluations of certain approximation polynomial $p^{(2)}\in {\rm span}(\bbV_n^{(2)})$ easy. In fact, assume $p^{(2)}$ is represented  as $p^{(2)}(\bx)= \sum_{j=1}^gc_j\xi^{(2)}_j(\bx) \in{\rm span}(\bbV_n^{(2)})$.
Then the coefficient vector $\bc=[c_1,\dots,c_g]^{\T}\in \bbR^{g}$ can first be computed accurately from a certain well-conditioned LS system by choosing a particular $G$-inner product (this is the goal (G2)); moreover, relying upon \eqref{eq:recuxi}, the resulting $p^{(2)}(\bx)$ can also readily generate evaluations at new nodes ${\mathcal S}=\{\bs_j\}_{j=1}^{\what m}$ (this is the goal (G1)). Particularly, \eqref{eq:basisfK} implies that 
\begin{equation}\nonumber%
	E:=[\xi_1^{(2)}({\mathcal S}),\dots,\xi_g^{(2)}({\mathcal S})]=[\kappa_1^{(2)}({\mathcal S}),\dots,\kappa_g^{(2)}({\mathcal S})]  \wtd R^{-1}\in\bbR^{\what m\wtd d\X g},
\end{equation}
and because we have the recurrence \eqref{eq:recuxi} with the known $\wtd R$, we can compute the  $i$th column of $E$ using the computed $\{\xi_j^{(2)}({\mathcal S})\}_{j=1}^{i-1}$. In particular, from $[\xi_1^{(2)}({\mathcal S}),\dots,\xi_g^{(2)}({\mathcal S})]\wtd R=[\kappa_1^{(2)}({\mathcal S}),\dots,\kappa_g^{(2)}({\mathcal S})]$ and $\kappa_i^{(2)}({\mathcal S})= {\bf S}_{u_i}^{(2)} \xi^{(2)}_{s_i}({\mathcal S})$, it holds that 
$$
\xi_i^{(2)}({\mathcal S})=\frac{1}{[\wtd R]_{i,i}}\left({\bf S}_{u_i}^{(2)}\,\xi_{s_i}^{(2)}({\mathcal S})-\sum_{j=1}^{i-1}\xi_j^{(2)}({\mathcal S}) [\wtd R]_{j,i}\right).
$$
This evaluation process (MV+G-A(E)) for computing $E$ is summarized in Algorithm \ref{alg:MCV+A-E}, where, similar to the matrix ${\bf X}_{u_i}^{(2)}$ on ${\mathcal X}$, ${\bf S}_{u_i}^{(2)}$ is generated from the matrix ${X}_{u_i}^{(2)}$ in \eqref{eq:recurrence0} in the Kronecker product fashion on nodes ${\mathcal S}$.    
Finally, 
\begin{equation}\label{eq:evaluationp}
	p^{(2)}({\mathcal S})=\left[\begin{array}{c}p(\mathcal S) \\\hline\partial_{1}p(\mathcal S)\\\vdots \\\partial_{d}p(\mathcal S)  \\\hline\partial_{1,1}p(\mathcal S)\\\vdots\\\partial_{d,d}p(\mathcal S)\end{array}\right]=E\bc.
\end{equation}
The overall complexity of flops of the evaluation procedure MV+G-A(E) is $O\left(  t^2 \widehat{m} \widetilde{d} \right).$

\begin{algorithm}[H] 
\caption{MV+G-A(E): Evaluation stage of the Multivariate confluent Vandermonde with $G$-Arnoldi} \label{alg:MCV+A-E}
\begin{algorithmic}[1]
\renewcommand{\algorithmicrequire}{\textbf{Input:}}
\renewcommand{\algorithmicensure}{\textbf{Output:}}
\REQUIRE $\{\bs_j\}_{j=1}^{\what{m}}\subset\bbR^d$,   the {\it grevlex} ordering basis   $\{\varphi_i\}_{i=1}^g$ of $\bbP_n^{d,\rm tol}$, and $\wtd{R}\in\bbR^{t\X t}$ from Algorithm \ref{alg:MCV+A-F}.
\ENSURE $E\in\bbR^{\what m\wtd d\X t}$.
\STATE $[E]_{:,1} \leftarrow 1/[\wtd{R}]_{1,1}$
\FOR{$i=2:t$}
\STATE pick the smallest ${s_i}\in [d]$ such that $\exists u_i\in [d]$ satisfying  $\varphi_i = x_{u_i}\varphi_{s_i}$
\STATE $\bw\leftarrow {\bf S}_{u_i}^{(2)}  [E]_{:,s_i}$
\STATE $\bw \leftarrow \bw-[E]_{:,1:i-1}[\wtd{R}]_{1:i-1,i}$
\STATE $[E]_{:,i} \leftarrow \bw/[\wtd{R}]_{i,i}$
\ENDFOR
\end{algorithmic}
\end{algorithm}

\section{Applications of MV+G-A and numerical results}\label{sec_apps} 
In this section, we   first state a general framework to specify the $G$-inner product in MV+G-A, and then provide numerical results of several particular applications. The numerical tests are carried out in MATLAB 2023b  on a 16-inch Macbook Pro with 2.6 GHz Intel Core i7 and 16Gb memory.

\subsection{The application-dependent $G$-inner product}\label{subsec:Ginner}
Given a compact domain $\Omega\subseteq \bbR^d$ and $f:\Omega\rightarrow \bbR^s ~(s\ge 1)$, we consider to find a function $u:\Omega\rightarrow \bbR$ subject to ${\cal L} u=f$ on $\Omega$, or more generally, the following $L_2$-integral minimization 
\begin{equation}\label{eq:contLS}
\min_{u}\|{\cal L} u-f\|_{{L_2(\Omega)}},
\end{equation}
where ${\cal L}$ is a linear (differential) operator involves the function value, and possibly the first-order and second-order partial derivatives. Let  ${\cal X}\subseteq \Omega$ be a set of sampled nodes. We aims at obtaining a multivariate polynomial $p\in \bbP_n^{d,{\rm tol}}$ to approximate the underlying $u$. As stated before, the purpose of MV+G-A is to find  basis functions $\{\xi_j(\bx)\}_{j=1}^g$ of $\bbP_n^{d,{\rm tol}}$ so that $\bc$ in 
\begin{equation}\label{eq:apppoly}
	p(\bx)= \sum_{j=1}^gc_j\xi_j(\bx) \in\bbP_n^{d,\rm tol},~~\bc=[c_1,\dots,c_g]^{\T}
\end{equation}
 can be computed by a well-conditioned discrete LS problem related to \eqref{eq:contLS}. Suppose ${\cal L}$ involves up to the second-order partial derivatives. Associated with ${\cal X}$ and $\{\xi_j(\bx)\}_{j=1}^g$ is the basis matrix $Q^{(2)}$ given in \eqref{eq:Qphi}.  Since ${\cal L}$ is a linear differential operator, there is a matrix ${\bf L}$ so that the spectral discretization of ${\cal L} p$ on ${\cal X}$ can be represented by 
$
{\bf L}Q^{(2)}\bc,
$
i.e., $${\cal L} u\approx {\cal L} p \xlongrightarrow{\text{ discretization on } {\cal X}}   {\bf L}Q^{(2)}\bc.$$ Therefore, we  have the following LS problem
\begin{equation}\label{eq:contLS}
\min_{\bc\in \bbR^g}\|{\bf L}Q^{(2)}\bc-\bb \|_{2}
\end{equation}
where $\bb$ is a vector resulting from discretizing the function $f$ on ${\cal X}$. 
 Let ${\bf A}={\bf L}Q^{(2)}$. MV+G-A obtains $\{\xi_j(\bx)\}_{j=1}^g$ so that 
 $$
 I={\bf A}^{\T}{\bf A}=(Q^{(2)})^{\T}{\bf L}^{\T}{\bf L} Q^{(2)}=:(Q^{(2)})^{\T}G Q^{(2)},
 $$
 which implies that the basis matrix $Q^{(2)}$ is orthonormal w.r.t.  the matrix $G={\bf L}^{\T}{\bf L}.$ In MV+G-A of Algorithm \ref{alg:MCV+A-F}, finding $\{\xi_j(\bx)\}_{j=1}^g$ with ${\bf A}^{\T}{\bf A}=I$ (hence $\bc={\bf A}^{\T}\bb$) can be realized readily by specifying the inner product:
 $$
 \langle \by,\bz\rangle_G=\by^{\T}{\bf L}^{\T}{\bf L}\bz
 $$
 in Steps 9 and 16. In general, as our applications in the next subsection demonstrate, the matrix $G$ is diagonal or block diagonal which hence facilitates the implementation easily. Furthermore, the method avoids matrix inversion, relying solely on matrix-vector multiplications throughout the computation.
  
\subsection{Applications of MV+G-A}\label{subsec:apps}
To best describe our general least-squares model in Section \ref{subsec:Ginner} for particular applications, for the basis functions \eqref{eq:recuxi} resulting from MV+G-A and a set of nodes ${\mathcal   X}=\{\bx_{j}\}_{j=1}^{m}$,  we introduce the notation 
\begin{subequations}\nonumber
	\begin{align} \label{eq:Qnodes1}
		Q^{[0]}({\mathcal   X})&:=\left[\begin{array}{ccc} \xi_1(\bx_1) & \cdots &  \xi_g(\bx_1) \\\vdots & \cdots & \vdots\\ \xi_1(\bx_m) & \cdots &  \xi_g(\bx_m) \end{array}\right]\in \bbR^{m\times g},\\\label{eq:Qnodes2}
		Q_{i}^{[1]}({\mathcal   X})&:=\left[\begin{array}{ccc}\partial_i\xi_1(\bx_1) & \cdots & \partial_i\xi_g(\bx_1) \\\vdots & \cdots & \vdots\\\partial_i\xi_1(\bx_m) & \cdots & \partial_i\xi_g(\bx_m) \end{array}\right]\in \bbR^{m\times g},~~i\in [d],~~{\rm and}\\\label{eq:Qnodes3}
		Q_{i,k}^{[2]}({\mathcal   X})&:=\left[\begin{array}{ccc}\partial_{i,k}\xi_1(\bx_1) & \cdots & \partial_{i,k}\xi_g(\bx_1) \\\vdots & \cdots & \vdots\\\partial_{i,k}\xi_1(\bx_m) & \cdots & \partial_{i,k}\xi_g(\bx_m) \end{array}\right]\in \bbR^{m\times g},~~i,k\in[d].
	\end{align}
\end{subequations}
Thus,   we can write the LS model of \eqref{eq:contLS} as
\begin{equation}\label{eq:generalmodel}
	\min_{\bc\in\bbR^{g}}\left\|{\bf A}\left(Q^{[0]}({\mathcal X}_0), \{Q_i^{[1]}({\mathcal X}_i)\}_{i=1}^d, \{Q_{i,k}^{[2]}({\mathcal X}_{i,k})\}_{1\le i\le k\le d}\right)\bc-\bb\right\|_2,
\end{equation}  
for  the desired multivariate polynomial $p(\bx)\approx u(\bx)$ in \eqref{eq:apppoly}.
Note that the vector $\bb$ contains the available information of the function $f$, and ${\mathcal X}_i\subseteq{\mathcal X},~{\mathcal X}_{i,j}\subseteq{\mathcal X}$ are all subsets of the original nodes  ${\mathcal X}$;
the coefficient matrix 
\begin{equation}\nonumber
	{\bf A}={\bf A}\left(Q^{[0]}({\mathcal X}_0), \{Q_i^{[1]}({\mathcal X}_i)\}_{i=1}^d, \{Q_{i,k}^{[2]}({\mathcal X}_{i,k})\}_{1\le i\le k\le d} \right)={\bf L}Q^{(2)}
\end{equation} 
is orthonormal, which depends on the function values and certain partial derivatives evaluated at particular nodes\footnote{The case ${\mathcal X}_i=\emptyset$ implies that the first-order partial derivative w.r.t $x_i$ is not involved explicitly in the LS problem \eqref{eq:generalmodel}.}.

The choice of $\mathcal{X} \subseteq \Omega$ is closely tied to the sample complexity and point distribution needed for accurate least-squares approximation, playing a pivotal role within the MV+G-A framework. For the special case where $\mathcal{L}u = u$, extensive literature exists on near-optimal sampling strategies for least-squares polynomial approximation (e.g., \cite{adca:2020,adcd:2022,adhu:2020,adsh:2023,codl:2013,comi:2017,zhna:2023}). In particular, \cite{zhna:2023}  demonstrates that for a broad class of domains, MV+G-A (which, when $\mathcal{L}u = u$, reduces to the multivariate V+A method \cite{aukl:2021,hoka:2020,zhna:2023}) achieves near-optimal approximation with $m = O(g^2)$ equally spaced samples, or $m = O(g^2 \log g)$ random samples. Further improvements are possible with weighted approaches: Weighted least squares \cite{comi:2017} and weighted MV+A \cite{zhna:2023}  can reduce sample complexity to $m = O(g \log g)$, while retaining near-optimal accuracy.
While the selection of $\mathcal{X}$ in a given domain $\Omega$ remains a critical research question in the general MV+G-A framework, we do not address it in this work.

The primary objective of the following examples is to illustrate how a particular application, given a set of nodes $\mathcal{X}$ in a domain $\Omega$, determines both the discretization matrix $\mathbf{L}$ and the $G$-inner product within the MV+G-A framework. Notably, all associated coefficient matrices ${\bf A}$ in \eqref{eq:generalmodel} exhibit a numerically favorable condition number of 1.

\subsubsection{The multivariate Hermite least-squares problem and a node-specific orthogonalization}\label{subsec:Hermit}
As our first example to illustrate the framework in Section \ref{subsec:Ginner}, we apply MV+G-A to compute the multivariate Hermite least-squares problem. 
\begin{example}\label{eg:Hermite1}
	{For a simple function \cite{zhna:2023}
		\begin{equation}\label{eq:egf1}
			f(\bx) = \sin(x_1x_2), \quad \bx\in {\mathcal D} = 
			\{ (x_1,x_2) | x_1^{2} + x_2^{2} \leq 1\}
		\end{equation} 
		on the unit disk ${\mathcal D}$, we use the mesh generator {\sf distmesh} \cite{pest:2004} 
		to get the points $\mathcal X $ on the uniform mesh. In Figure \ref{fig:ex4_1_1}, we plot these nodes 
		including $m_0 = 120$ points (in blue) ${\mathcal X}_0=\{\bx_j\}_{j=1}^{m_0}$ inside the domain ${\mathcal D}$ and $m_1= 42$ boundary points (in red) ${\mathcal X}_1=\{\bx_j\}_{j=m_0+1}^{m_0+m_1}$. As a toy example, we provide function values  for ${\mathcal X}_0=\{\bx_j\}_{j=1}^{m_0}$, while provide both function values as well as the associated partial derivatives w.r.t. $x_1$ and $x_2$   for the boundary points  ${\mathcal X}_1=\{\bx_j\}_{j=m_0+1}^{m_0+m_1}$.  Thus, for MV+G-A(F), ${\mathcal X}={\mathcal X}_0\cup {\mathcal X}_1=\{\bx_j\}_{j=1}^m$ where $m=m_0+m_1= 162$. 
		To evaluate the accuracy of the computed polynomial, we choose $\what m=560$ refined nodes in ${\mathcal D}$ in the evaluation stage MV+G-A(E). 
	} 
\end{example} 
\begin{figure}[h!!!]
	\centering
	\epsfig{file = 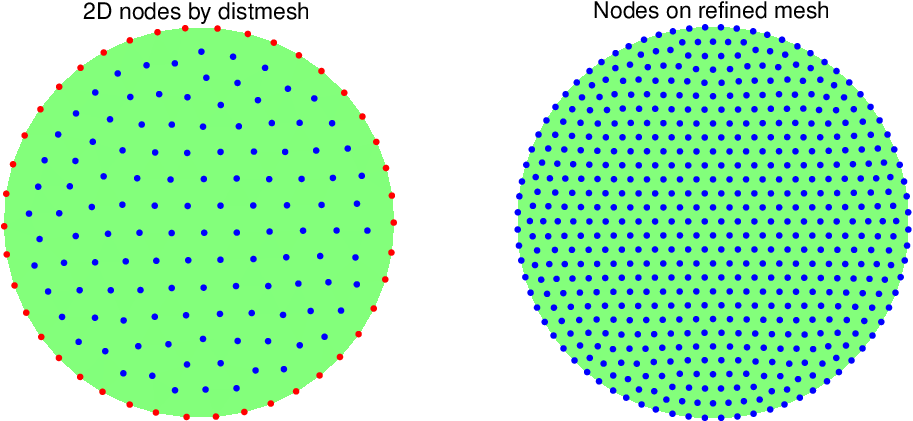, width = 6cm}
	\caption{Uniform mesh on unit disk  by {\sf distmesh.}  (left) fitting nodes including 
		(in red) boundary points ($m_1=42$) and (in blue) interior points ($m_0=120$); (right) 
		nodes for evaluations ($\what m=560$).} 
	\label{fig:ex4_1_1}
\end{figure}

To obtain the coefficient $\bc$ in \eqref{eq:generalmodel}, we have the right-hand side 
\begin{equation}\nonumber
	\bb=[f(\bx_1),\dots,f(\bx_m),\partial_1f(\bx_{m_0+1}),\dots,\partial_1f(\bx_{m}),\partial_2f(\bx_{m_0+1}),\dots,\partial_2f(\bx_{m})]^{\T}\in \bbR^{m+2m_1}.
\end{equation}
Note that the associated linear operator ${\cal L}$ does not involve the second order partial derivatives, and hence $Q^{(1)}=\left[\begin{array}{c}Q_0 \\Q_1\end{array}\right]$ should be used instead of $Q^{(2)}$ \eqref{eq:Qphi} in the model \eqref{eq:contLS}. The discretization matrix ${\bf L}$ is
$$
{\bf L}= \left[\begin{array}{c|cc|cc}I_{m} &  &  &  &  \\ \hline &  {0}_{m_1\times m_0} & I_{m_1} &  &  \\ \hline &  &  &   {0}_{m_1\times m_0}& I_{m_1}\end{array}\right]\in \bbR^{(m+2m_1) \times 3m},
$$
and   the coefficient matrix in the LS \eqref{eq:generalmodel} is 
$$
{\bf A}\left(Q^{[0]}({\mathcal X}_0), \{Q_i^{[1]}({\mathcal X}_1)\}_{i=1}^2\right)={\bf L}Q^{(1)}={\bf L}
\left[\begin{array}{c}
Q^{[0]}({\mathcal X}) \\ \hline Q_1^{[1]}({\mathcal X}_0)\\
	    Q_1^{[1]}({\mathcal X}_1)\\\hline
	      Q_2^{[1]}({\mathcal X}_0)\\
	     Q_2^{[1]}({\mathcal X}_1) \end{array}\right] = \kbordermatrix{ &\sss g \\
	\sss m & Q^{[0]}({\mathcal X})    \\ 
	\sss m_1 &  Q_1^{[1]}({\mathcal X}_1)\\
	\sss m_1 &  Q_2^{[1]}({\mathcal X}_1)}  \in \bbR^{(m+2m_1)\times g},
$$
which is orthonormal if we use the positive semi-definite matrix 
$$
G={\bf L}^{\T}{\bf L}=\diag(\underbrace{I_{m}}_{{\rm for~} p},\underbrace{{0}_{m_0 \X m_0},I_{m_1}}_{{\rm for} ~\partial_{x_1}p},\underbrace{{0}_{m_0\X m_0},I_{m_1}}_{{\rm for} ~\partial_{x_2}p})\in \bbR^{3m\times 3m}
$$
to define the $G$-inner product \eqref{eq:Ginner}. 
Accordingly, for a pair of vectors $\by,\bz\in \bbR^{3m}$, the $G$-inner product is 
\begin{align*}
	\langle \by,\bz\rangle_G&=  \by^{\T}{\bf L}^{\T}{\bf L}\bz= \left[\begin{array}{cc} \by_{1:m}\\   {0}_{m_0}\\\by_{m+m_0+1:2m}\\   \bzs_{m_0}\\\by_{2m+m_0+1:3m}\end{array}\right]^{\T} \left[\begin{array}{cc} \bz_{1:m}\\   \bzs_{m_0}\\\bz_{m+m_0+1:2m}\\   \bzs_{m_0}\\\bz_{2m+m_0+1:3m}\end{array}\right]\\
	&=\by_{1:m}^{\T}\bz_{1:m}+\by_{m+m_0+1:2m}^{\T}\bz_{m+m_0+1:2m}+\by_{2m+m_0+1:3m}^{\T}\bz_{2m+m_0+1:3m}.
\end{align*}
It is worth mentioning that the $G$-Arnoldi process with this $G$-inner product is just the {\it sub-orthogonalization  Arnoldi process} presented in \cite[Section 9]{lirc:2004rept}. In particular, in our case, this sub-orthogonalization  Arnoldi process is a {\it node-specific orthogonalization} in that only nodes with available partial derivatives are involved in the orthogonalization.

Set $n=10$ to have $g=66$. Applying MV+G-A(F), we first compute the 
coefficient vector $\bc$, and then compute via MV+G-A(E) the function values at 
new nodes $\{\bs_j\}_{j=1}^{560}$ (right subfigure in  Figure \ref{fig:ex4_1_1}).  
In Figure \ref{fig:ex4_1_2}, we plot the  approximation $p$ of $f = \sin(\bx)$, 
as well as the errors $f(x_1,x_2)-p(x_1,x_2)$ at the refined new points. 

\begin{figure}[t!!!]
	\centering
	\epsfig{file = 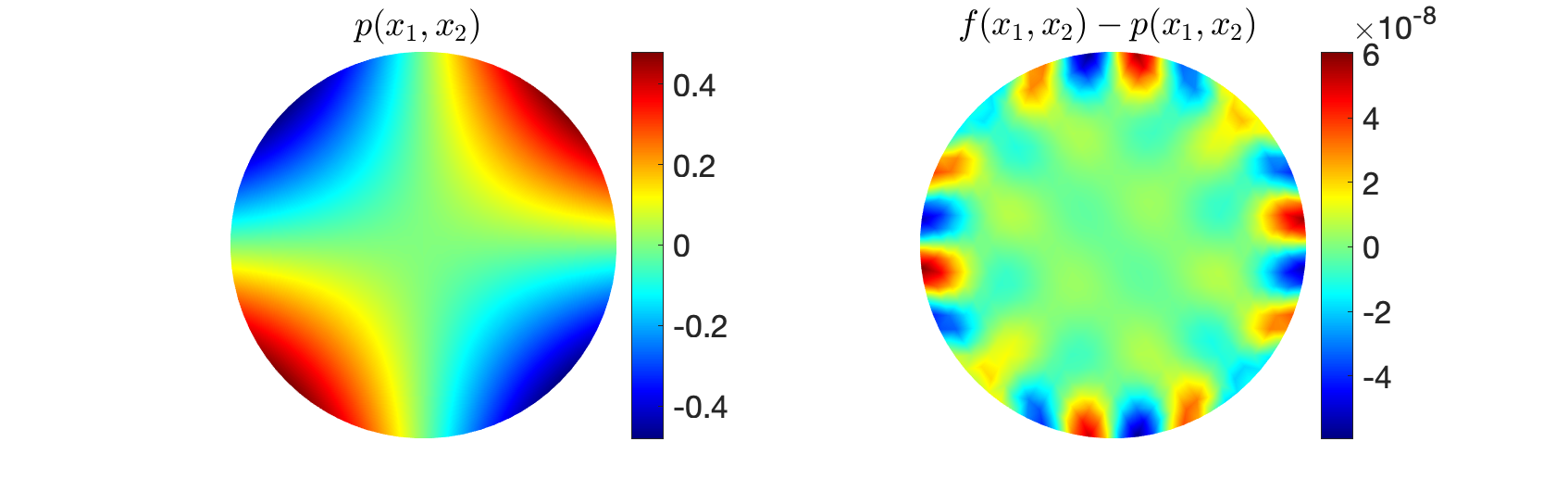,width = 11cm }
	\caption{(left) Hermite LS approximation of $f(\bx)=\sin(x_1x_2)$ in \eqref{eq:egf1}, and  (right) error of the Hermite LS approximation with $n = 10$ and $g = 66$ on  nodes $\{\bs_j\}_{j=1}^{560}$.} 
	\label{fig:ex4_1_2}
\end{figure}

 \begin{example}\label{exam:3d-2}
We next consider a 3-dimensional  Hermite least-squares problem with 
\begin{align*}
   f(x_{1},x_{2},x_{3})  
& =   x_{1}^2 + 2x_{2}^{2} +2x_{3}^{2} + \frac{1}{2}[
\sin(\pi x_{1} ) + \sin(\pi x_{2} ) + \sin(\pi x_{3} ) ] + \sin(x_{1}x_{2}x_{3}), \\
  \quad  \boldsymbol{x}
 \in \Omega&=[-1,1]^3:=
 \{(x_{1},x_{2},x_{3}) \; | \;  | x_{1} | \leq 1,   | x_{2} | \leq 1,   | x_{3} | \leq 1   \}.
\end{align*}  
To form a test problem, in $\Omega$, we randomly choose  $m = 2000$ nodes  ${\cal X}=\{\bx_j\}_{j=1}^m$ (the red points in the left subfigure of Figure \ref{fig:3d-2-1}), and provide values of $f(\bx_j)$,   $\partial_{{1}}f(\bx_j)$ and  $\partial_{{3}}f(\bx_j)$ for $\bx_j\in {\cal X}$. The Hermite least-squares problem then can be reformulated as \eqref{eq:contLS} with 
\begin{equation}\nonumber
	\bb=[f(\bx_1),\dots,f(\bx_m),\partial_1f(\bx_{1}),\dots,\partial_1f(\bx_{m}),\partial_3f(\bx_{1}),\dots,\partial_3f(\bx_{m})]^{\T}\in \bbR^{3m},
\end{equation}
$$
{\bf L}= \left[\begin{array}{c|cc|c|c}I_{m} &  &  &  &  \\ \hline &  {I}_{m} &   &  &  \\ \hline &  &  &   {0}_{m \times m }& I_{m}\end{array}\right]\in \bbR^{3m \times 4m}, ~{\bf A}= {\bf L}Q^{(1)} = \kbordermatrix{ &\sss g \\
	\sss m & Q^{[0]}({\mathcal X})    \\ 
	\sss m  &  Q_1^{[1]}({\mathcal X})\\
	\sss m &  Q_3^{[1]}({\mathcal X})}  \in \bbR^{3m\times g}.
$$
Associated  with $G={\bf L}^{\T}{\bf L}=\diag(I_m,I_m,0_{m\times m},I_m)$ is the  $G$-inner product for  $\by=[\by_1^{\T},\dots,\by_4^{\T}]^{\T},\bz=[\bz_1^{\T},\dots,\bz_4^{\T}]^{\T}\in \bbR^{4m}$   given by
$$\langle \by,\bz\rangle_G =  \by^{\T}{\bf L}^{\T}{\bf L}\bz=\by_1^{\T}\bz_1+\by_2^{\T}\bz_2+\by_4^{\T}\bz_4,~\by_i,\bz_i\in \bbR^m, ~1\le i\le 4.$$
Applying  MV+G-A with $n=13$, we plot the isosurface  of Hermite least-squares  approximation 
 $ p(x_1,x_2,x_3) = 0.5$ in the middle subfigure of Figure \ref{fig:3d-2-1}; the right subfigure in Figure \ref{fig:3d-2-1} displays the approximation errors $f(x_1,x_2,x_3)-p(x_1,x_2,x_3)$ computed at $32^3 = 32768$ uniform grid points, with the point indices shown along the $x$-axis.
\begin{figure}[htbp]
\centering
\hskip -9mm
\epsfig{file = 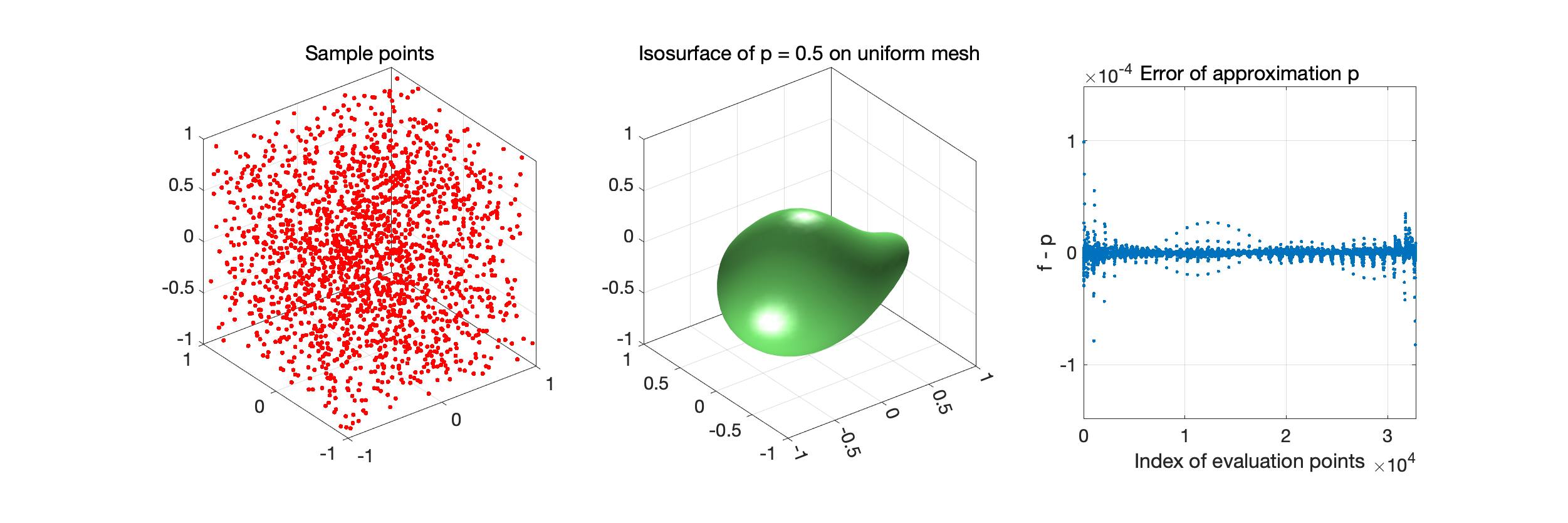, width = 16cm}
\caption{(left) random sample points in $\Omega=[-1,1]^{3}$; (middle)  isosurface of Hermite least-squares  approximation 
 $ p(x_1,x_2,x_3) = 0.5$ with $n=13$;   (right) errors $f(x_1,x_2,x_3) - p(x_1,x_2,x_3)$ at $32^3 = 32768$ uniformly sampled nodes (indexed on the $x$-axis).
 }
\label{fig:3d-2-1}
\end{figure}
\end{example}

\subsubsection{Function values to  divergence and Laplace approximations}\label{subsec:graddiv}
As illustrated in \cite[Example 4]{nizz:2023} for the univariate case, by constructing a polynomial approximation $p$ using only the function values,  one can apply V+A to approximate the values of (high-order) derivatives of the underlying function $f$. This applies similarly to MV+G-A, which is useful for the situation when only values of the   function or its certain partial derivatives at a discrete set of nodes are available.
  
\begin{example}\label{eg:graddiv1}
	{For a given $f(\bx)=e^{-3(x_1^{2} + x_1x_2+ x_2^{2})}$, we first compute the 2D interpolation $p(\bx)$ on the Padua points \cite{camv:2005} ${\mathcal X}=\{\bx_j\}_{j=1}^m$ with $m=561$. These nodes are plotted in Figure \ref{fig:ex4_2_1}. To apply MV + G-A(F) in this case, corresponding to the discrete matrix ${\bf L}=[I_m,{0}_{m},{0}_{m}, {0}_{m},{0}_{m},{0}_{m}]\in \bbR^{m\times 6m}$, we use the inner-product induced by a matrix $G=\diag(I_m,{0}_{m}, {0}_{m},{0}_{m},{0}_{m},{0}_{m})\in \bbR^{6m\times 6m}$ where $0_m=0_{m\X m}$.  In other words, the fitting stage for computing the coefficient vector $\bc$ in \eqref{eq:generalmodel} is just the multivariate V+A \cite{aukl:2021,hoka:2020,zhna:2023} for which the LS orthonormal coefficient matrix is
		$$
		{\bf A}\left(Q^{[0]}({\mathcal X})\}\right)={\bf L}Q^{(2)}= Q^{[0]}({\mathcal X}) \in \bbR^{m\times g}.
		$$
		After computing $\bc$, the evaluations of the partial derivatives of $\partial_{1} p(\bx),~ \partial_{2} p(\bx),~  \partial_{1,1} p(\bx)$ and $\partial_{2,2} p(\bx)$ at new nodes $\{\bs_j\}_{j=1}^{\what m}$ ($\what m=41^2=1681$) are computed via \eqref{eq:evaluationp}.  In this test, we set $n = 32$ to have $g=561$ which is equal to the number $m$ of Padua points.
		Figure \ref{fig:ex4_2_2} shows the errors of the interpolant $p$,  the divergence  
		$\mathop{\rm div} p$, and the Laplace $\Delta p=\partial_{1,1}p(\bx)+\partial_{2,2} p(\bx).$
	}
	\begin{figure}[htbp]
		\centering
		\epsfig{file = 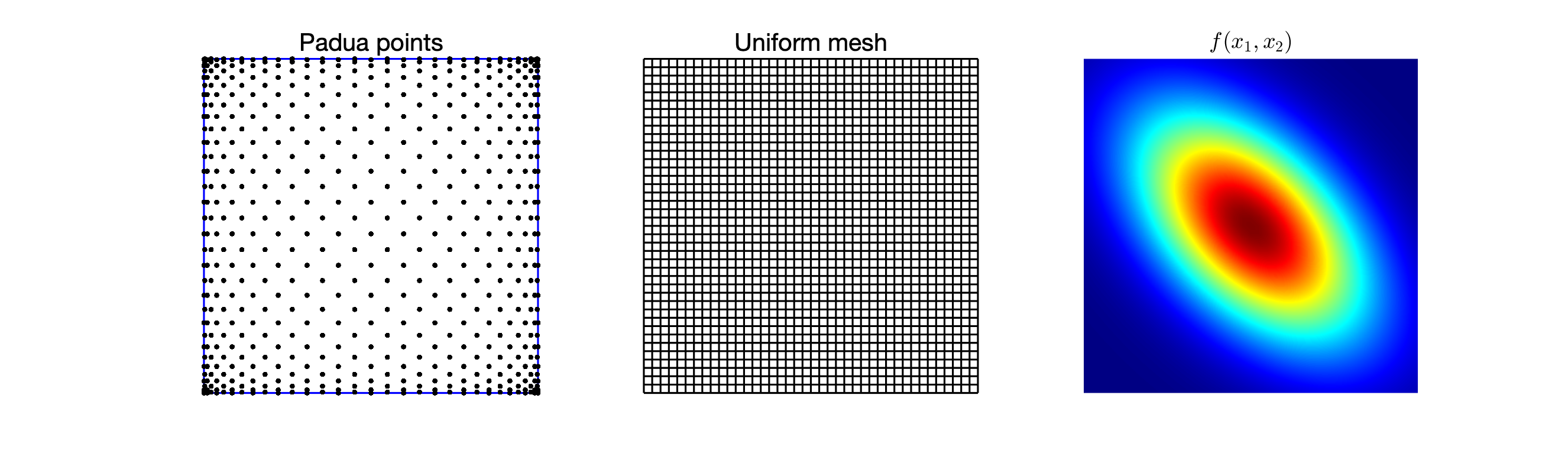,width = 12cm}
		\caption{ (left) fitting nodes (Padua points);  (middle) evaluations at uniform nodes; (right) test function $f(\bx)=e^{-3(x_1^{2} + x_1x_2+ x_2^{2})}$.} \label{fig:ex4_2_1}
	\end{figure}
	
	\begin{figure}[htbp]
		\centering
		\epsfig{file = 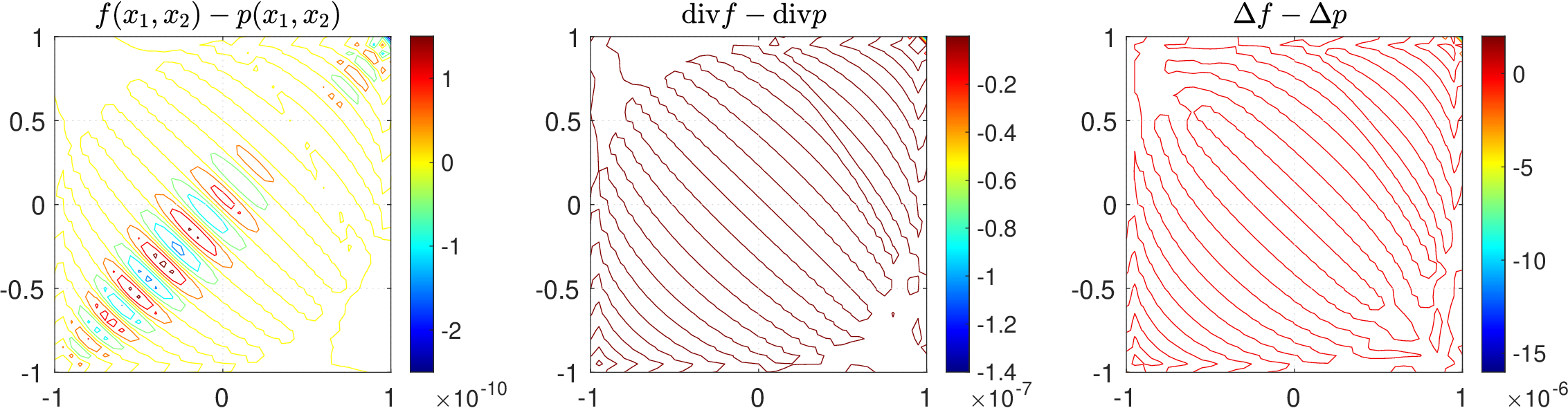,width = 12cm }
		\caption{ (left) interpolant of $f(\bx)=e^{-3(x_1^{2} + x_1x_2+ x_2^{2})}$ on 
			Padua points; (middle) error of $\mathop{\rm div} p(\bx)$ on new points  $\{\bs_j\}_{j=1}^{1681}$; 
			(right) error of $\Delta p(\bx)$ on new points $\{\bs_j\}_{j=1}^{1681}$.} 
		\label{fig:ex4_2_2}
	\end{figure}
\end{example}

\begin{example}\label{exam:3d-1}
We   consider a three-variable function defined on the unit sphere:
\[ f(x_{1},x_{2},x_{3}) = 5x_{1}x_{2}x_{3} \cos(x_{1}-x_{2}+2x_{3}), \quad  \boldsymbol{x}
 \in \Omega=\bbS^{2}:=
\{(x_{1},x_{2},x_{3}) \; | \; x_{1}^{2} + x_{2}^{2} + x_{3}^{3} = 1\}. \]
Choosing $m = 1000$ random  nodes ${\cal X}=\{\bx_j\}_{j=1}^m$ (red nodes in the left subfigure of Figure \ref{fig:ex4_4_2}), and providing values  $f(\bx_j)$ and $\partial_{{2}}f(\bx_j)$ for  $\bx_j\in {\cal X}$, we obtain an approximation polynomial $p(x_1,x_2,x_3)$ with $n=13$ of $f(x_{1},x_{2},x_{3})$ from solving  \eqref{eq:contLS} where 
\begin{equation}\nonumber
	\bb=[f(\bx_1),\dots,f(\bx_m),\partial_2f(\bx_{1}),\dots,\partial_2f(\bx_{m})]^{\T}\in \bbR^{2m},
\end{equation}
$$
{\bf L}= \left[\begin{array}{cccc }I_{m} & 0_m & 0_m & 0_m\\0_m  & 0_m  &    {I}_{m} &0_m  \end{array}\right]\in \bbR^{2m \times 4m}, ~{\bf A}= {\bf L}Q^{(1)} = \kbordermatrix{ &\sss g \\
	\sss m & Q^{[0]}({\mathcal X})    \\ 
	\sss m  &  Q_2^{[1]}({\mathcal X})
 }  \in \bbR^{2m\times g}.
$$
The matrix   $G={\bf L}^{\T}{\bf L}=\diag(I_m,0_{m},I_m,0_{m})\in \bbR^{4m\times 4m}$ and for $\by=[\by_1^{\T},\dots,\by_4^{\T}]^{\T},\bz=[\bz_1^{\T},\dots,\bz_4^{\T}]^{\T}\in \bbR^{4m}$, 
$$\langle \by,\bz\rangle_G =  \by^{\T}{\bf L}^{\T}{\bf L}\bz=\by_1^{\T}\bz_1+\by_3^{\T}\bz_3,~\by_i,\bz_i\in \bbR^m, ~1\le i\le 4.$$
After computing $p$, we  then apply \eqref{eq:evaluationp} and MV+G-A(E) of Algorithm \ref{alg:MCV+A-E} to evaluate the partial derivatives of $\partial_{1} p(\bx),\partial_{2} p(\bx), \partial_{3} p(\bx)$ and compute the rotational surface gradient\footnote{For a differentiable function $p: \bbR^3 \rightarrow \bbR$ and its restriction $p_{|_{\bbS^2}}:\bbS^2\rightarrow \bbR$ on $\bbS^2$, the rotational surface gradient at $\bx\in \bbS^2$  is given by 
$\bn \times {\rm grad}\,p_{|_{\bbS^2}}\in T_{\bx} \bbS^2$, 
where
$T_{\bx} \bbS^2$ is the tangent space of $\bbS^2$ at $\bx \in \bbS^2$, $\bn = \bn(\bx)$ is the unit normal vector at $\bx$ and
${\rm grad}\,p_{|_{\bbS^2}}$ is the gradient of  $p_{|_{\bbS^2}}:\bbS^2\rightarrow \bbR$. Noting that 
$ {\rm grad}\,p_{|_{\bbS^2}} = (I_3 - \bn \bn^{\T}) \nabla p,$  we have  
$\bn \times {\rm grad}\,p_{|_{\bbS^2}} = \bn \times \nabla p - \underbrace{(\bn \times \bn)}_{=\bzs}(\bn^{\T} \nabla p) = \bn \times \nabla p,$
yielding  \eqref{eq:curlp}.
See also the function {\tt curl} in {\tt spherefun} of Chebfun package \url{https://www.chebfun.org}.}:

\begin{equation}\label{eq:curlp}
 {\bn} \times  \nabla p  = 
 \begin{vmatrix} 
{\tt i} & {\tt j}  & {\tt k} \\
 x_{1} & x_{2} & x_{3}
 \\
 \partial_{{1}}p & \partial_{{2}}p & \partial_{{3}}p
\end{vmatrix},~ \bn=[x_1,x_2,x_3]^{\T}\in \bbS^2
\end{equation}
at new nodes $\{\bs_j\}_{j=1}^{1806}$ obtained from the  uniform triangular mesh (mesh points in the right subfigure of Figure \ref{fig:ex4_4_2}). The middle subfigure in Figure \ref{fig:ex4_4_2} depicts the function values $p(\bs_j)$ using a color scale, alongside the vector field $\bn \times \nabla p$ evaluated at the new nodal points $\{ \bs_j \}_{j=1}^{1806}$. The right subfigure illustrates the corresponding errors $f(x_1,x_2,x_3)-p(x_1,x_2,x_3)$, also represented in a color scale.

 \end{example}
\begin{figure}[t!!!]
	\centering
	\hskip  -9mm
	\epsfig{file = 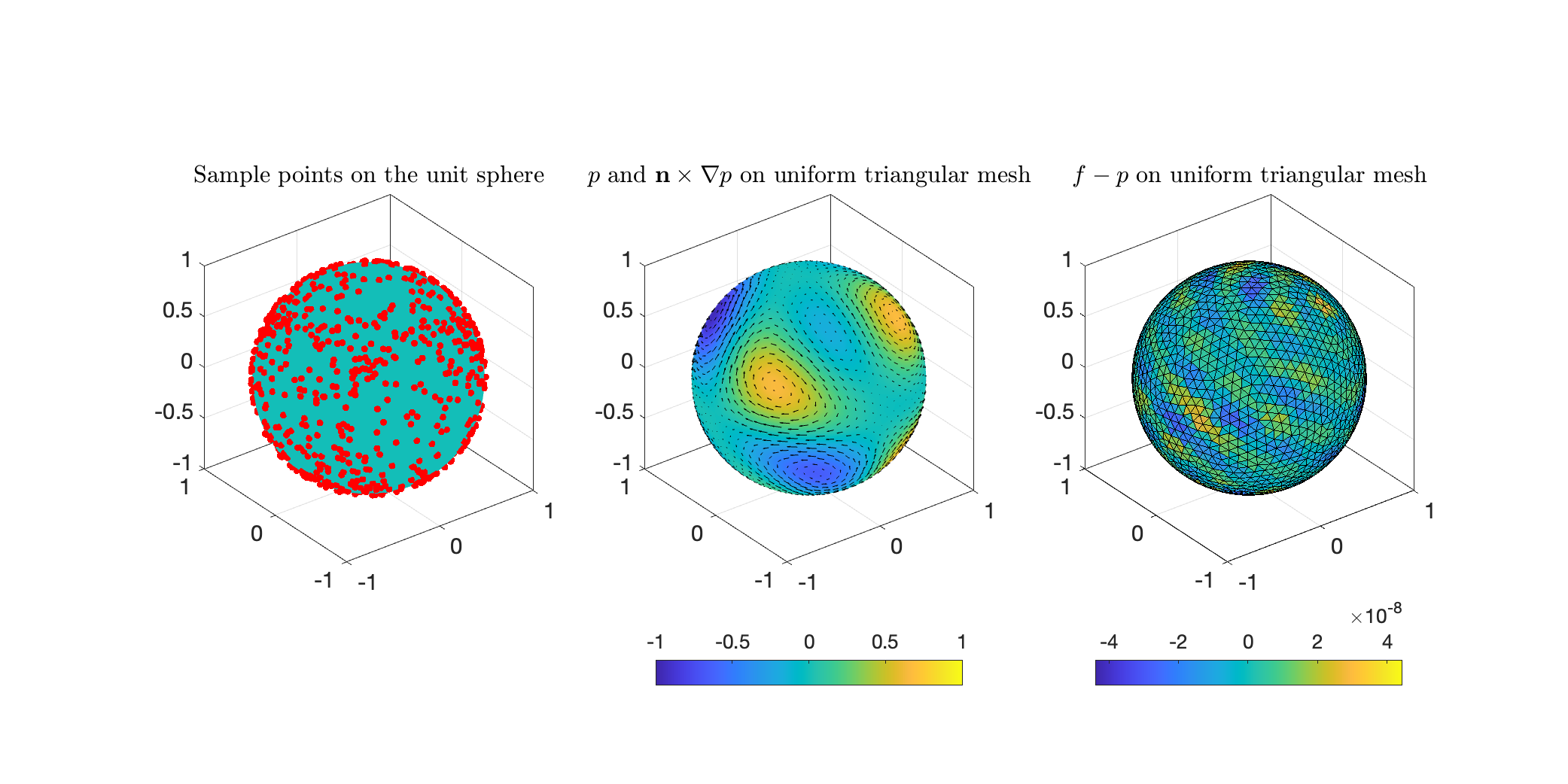,width = 16cm }
	\caption{(left) random sample points on the unit sphere $\bbS^2$; (middle)  evaluations of ${\bn}\times \nabla p$ at new nodes  $\{  \bs_j\}_{j=1}^{ 1806}$;  function values of $p(\bs_j)$ are displayed in a color scale;  (right) error of  Hermite LS approximation $p$ with $n = 13$.} 
	\label{fig:ex4_4_2}
\end{figure}

\subsubsection{Solving PDEs and PDE-specific orthogonalization}\label{subsec:PDE}
Next, we apply the MV+G-A framework to several PDE problems. While numerous efficient methods exist for various types of PDEs---particularly those defined on regular domains---MV+G-A may not be a very computationally competitive option. However, our primary objective is to demonstrate that MV+G-A provides a convenient alternative for solving linear PDEs on irregular domains, requiring only the matrix-vector multiplications in the multivariate confluent Arnoldi process and potentially addressing the ill-conditioned problem in the spectral collocation methods. We begin by applying MV+G-A to Poisson’s equation as an illustrative example.

\begin{example}\label{eg:PDE1}
	Consider 
	\begin{equation}\label{eq:poisson}
		\left\{\begin{array}{ll}&u(\bx)+\alpha(\bx) \Delta u(\bx)=f(\bx),~~\bx\in \Omega\subseteq\bbR^2, 
			\\&u(\bx)=h(\bx),~~\bx\in \partial\Omega,\end{array}\right.
	\end{equation} 
	where $\alpha(\bx)$ is a given function or a constant,   $f(\bx)$ is the right-hand side and $h(\bx)$ is the function for the Dirichlet  boundary condition of $u(\bx)$. As the first illustration for solving PDE, we assume $\alpha(\bx)\equiv \alpha$ is a constant. 
\end{example}

To apply MV+G-A, we   choose nodes ${\mathcal X}_0=\{\bx_j\}_{j=1}^{m_0}$ in $\Omega$ and nodes ${\mathcal X}_1=\{\bx_j\}_{j=m_0+1}^{m}$ on $\partial \Omega$ to construct a certain LS system in the form of \eqref{eq:generalmodel}. Let $m_1=m-m_0$ and $p$ given in \eqref{eq:apppoly} be the approximation of the solution  $u$. In this case, ${\bf L}=E_6P_{\alpha}$ is  the discretization matrix, where  
\begin{align} \label{eq:Palpha}
	P_{\alpha}=\diag(\underbrace{I_{m}}_{{\rm for~} p},\underbrace{{0}_{m\times m}}_{{\rm for~} \partial_{1}p},\underbrace{{0}_{m\times m}}_{{\rm for~} \partial_{2}p}, \underbrace{\alpha I_{m_0}, {0}_{m_1 \times m_1}}_{{\rm for~} \partial_{1,1}p}, \underbrace{{0}_{m\times m}}_{{\rm for~} \partial_{1,2}p},\underbrace{\alpha I_{m_0}, {0}_{m_1\times m_1}}_{{\rm for~} \partial_{2,2}p})\in \bbR^{6m\times 6m}
\end{align}
and $E_6=[I_m,\dots,I_m]\in \bbR^{m\X 6m}$. Thus  the resulting  LS system \eqref{eq:generalmodel} admits the following coefficient matrix 

\begin{align}\nonumber
	&{\bf A}\left(Q^{[0]}({\mathcal X}_0), \{Q_i^{[1]}({\mathcal X}_1),\{Q_{i,k}^{[2]}({\mathcal X}_{i,k})\}_{1\le i\le k\le 2}\right)={\bf L}Q^{(2)}\\\label{eq:APoisson}
	=& \left[\begin{array}{c} Q^{[0]}({\mathcal X}_0)\\ Q^{[0]}({\mathcal X}_1)\end{array}\right]+\left[\begin{array}{c} \alpha\left(Q_{1,1}^{[2]}({\mathcal X}_0)+Q_{2,2}^{[2]}({\mathcal X}_0)\right)\\  {0}_{m_1\times g}\end{array}\right] \in \bbR^{m\times g}
\end{align}
and   
$$
\bb=\left[\begin{array}{c}f({\mathcal X}_0) \\h({\mathcal X}_1)\end{array}\right] \in \bbR^{m}
$$
as the corresponding right-hand side. Accordingly,  the matrix $G$ is\footnote{For a general function $\alpha(\bx)$ in \eqref{eq:poisson}, we can simply replace the constant $\alpha$ with a diagonal matrix $\Upsilon=\diag(\alpha(\bx_1),\dots,\alpha(\bx_{m_0}))$ in \eqref{eq:Palpha}. This is illustrated in Example \ref{eg:PDE21}.}
\begin{equation}\label{eq:GPDE}
	G=(E_6P_{\alpha})^{\T}(E_6P_{\alpha})\in \bbR^{6m\X 6m}.
\end{equation}
For  $\by,\bz\in \bbR^{6m}$, the $G$-inner product used in Steps 9 and 16 in Algorithm \ref{alg:MCV+A-F}  is 
{\small 
	\begin{align*}
		&\langle \by,\bz\rangle_G=\by^{\T}{\bf L}^{\T}{\bf L}\bz\\
		=&\left(\by_{1:m}+  \left[\begin{array}{cc}  \alpha(\by_{3m+1:3m+m_0}+\by_{5m+1:5m+m_0})\\\mathbf{0}_{m_1}\end{array}\right]  \right)^{\T}\left(\bz_{1:m}+  \left[\begin{array}{cc}  \alpha(\bz_{3m+1:3m+m_0}+\bz_{5m+1:5m+m_0})\\\mathbf{0}_{m_1}\end{array}\right]  \right).
	\end{align*}
}

To demonstrate the result, we set the domain
\[ \Omega= \left\{ \bx=(x_1,x_2) | x_1^{2} + \frac14 x_2^{2} \leq 1\right\}   \setminus 
\left\{ (x_1,x_2) | x_1^{2} +x_2^{2}  \leq \frac14\right\},\]  
and  choose  $m_0= 504$ nodes in $\Omega$ and $ m_1= 126$ boundary nodes via {\sf distmesh} \cite{pest:2004}.
In order to demonstrate the high accuracy for the problem with a smooth solution, 
we set a test problem with exact known solution $u(\bx) = e^{x_1+\frac{x_2}{2}}$ and $\alpha = -0.1$, from which
the boundary function $h(\bx)$ and right hand side $f(\bx)$ can be computed directly from \eqref{eq:poisson}. With these conditions, we apply MV+G-A to compute polynomial approximant for the solution $u$ as well as the error $u(\bx)-p(\bx)$ at ${\mathcal X}$. 
Figure \ref{fig:ex4_3_1} gives results with $n = 22$ and $g = 276.$ We observe that the numerical solution $p$ provides highly accurate approximation for $u$, and moreover, it is computationally efficient and flexible. These features suggest the efficiency of  MV+G-A in solving \eqref{eq:poisson}.

\begin{figure}
	\centering
	\hskip -5mm
	\epsfig{file = 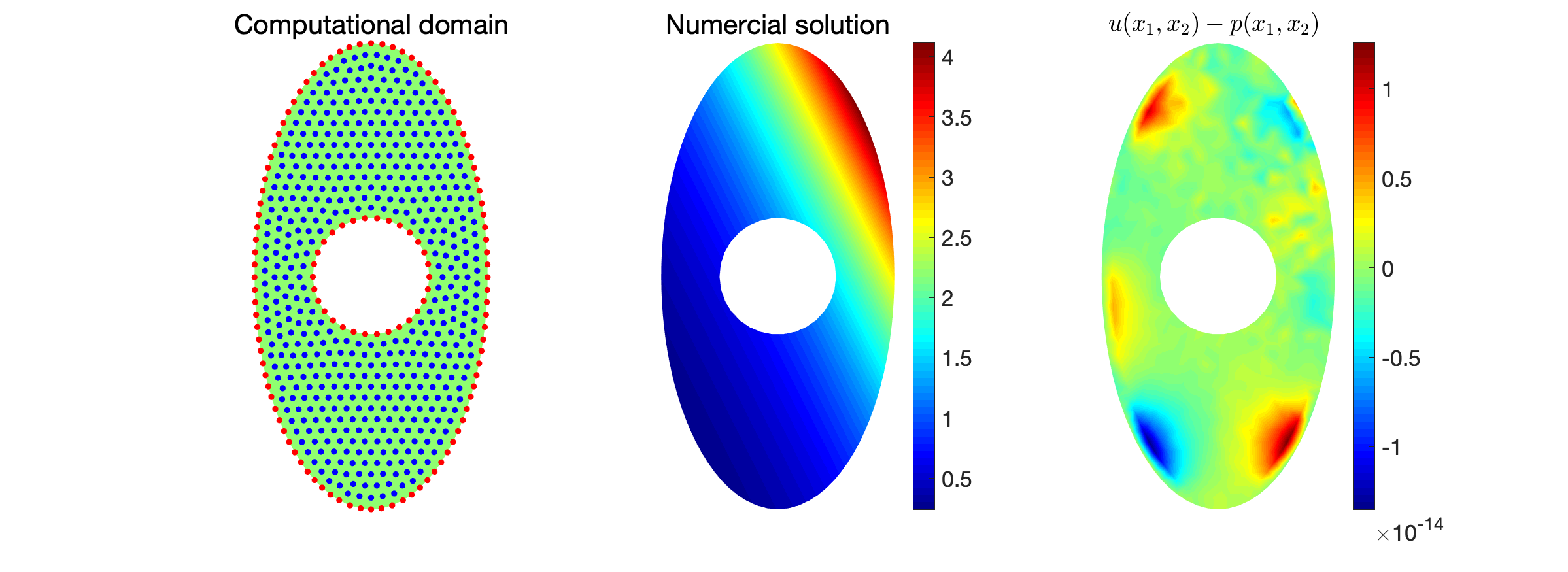,width = 14cm}
	\caption{ (left) fitting nodes;  (middle) numerical solution $p(\bx)$;  (right) error of $p(\bx)$.} 
	\label{fig:ex4_3_1}
\end{figure}

\begin{example}\label{eg:PDE21}
	To extend Example \ref{eg:PDE1}, we modify the constant $\alpha$ in \eqref{eq:poisson} as $\alpha(\bx)=-e^{-\|\bx\|_2^{2}}$ and consider again the solution for the Poisson equation \eqref{eq:poisson} with the same domain $\Omega$. Particularly, we set $h(\bx) \equiv 0$ on the boundary and $f(\bx) \equiv 1$ as the right hand side function.    The domain $\Omega$ is discretized  by {\sf distmesh} \cite{pest:2004} with $m_{0} = 3650$ and $m_{1} = 331$, and set  $n = 40 ~(g = 861)$ to generate a polynomial $p$ for the solution $u$. Based on the $G$-inner product specified in \eqref{eq:GPDE} by simply replacing the constant $\alpha$ with a diagonal matrix $\Upsilon=\diag(\alpha(\bx_1),\dots,\alpha(\bx_{m_0})),$ we arrive at an orthonormal coefficient matrix, analogous to \eqref{eq:APoisson}, for the LS problem \eqref{eq:generalmodel}. 
	To reflect the accuracy, because  we do not have the closed-form solution $u$, we compute
	the residual ${\mathcal  L} p - f$ as a measure of the accuracy, where ${\mathcal  L} u = u + \alpha \Delta u$. 
	Figure \ref{fig:ex4_3_2} plots the fitting nodes (left subfigure) in the domain $\Omega$,  the numerical solution (middle subfigure), and the residual (right subfigure)  at ${\mathcal X}$.  
	\begin{figure}[t!!!]
		\centering
		\hskip -5mm
		\epsfig{file = 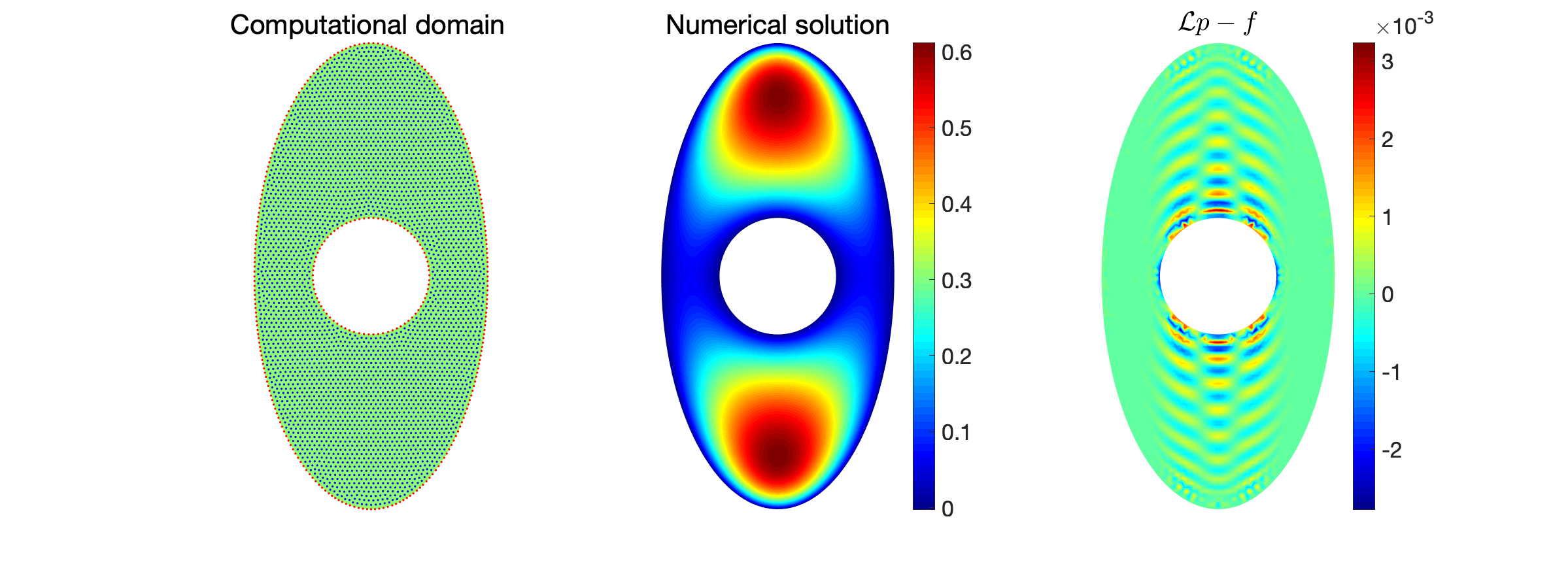,width = 14cm}
		\caption{ (left)  fitting nodes;   (middle) numerical solution $p(\bx)$;  (right)  residual ${\mathcal L}p - f$.} 
		\label{fig:ex4_3_2}
	\end{figure}
\end{example}

\begin{example}\label{eg:PDE3}
	As our last application for  demonstrating the flexibility of our method,  we consider  the Poisson equation with both the Dirichlet and Neumann conditions:
	\begin{equation}\label{eq:poissonMixBC}
		\left\{\begin{array}{ll}&u(\bx)+\alpha(\bx) \Delta u(\bx)=f(\bx),~~\bx\in \Omega\subseteq\bbR^2, 
			\\&u(\bx)=h_1(\bx),~~\bx\in \Gamma_1, \\
			& \nabla u\cdot  {\bn} = h_2(\bx),~~\bx\in \Gamma_2,
		\end{array} \right.
	\end{equation} 
	where $\partial\Omega = \Gamma_1\cup \Gamma_2$, $\bn=[n_1(\bx),n_2(\bx)]^{\T}$ denotes the outward normal unit vector to $\Gamma_2$ at $\bx\in \Gamma_2$, and $\nabla u(\bx)=[\partial_1 u(\bx),\partial_2 u(\bx)]^{\T}$ is the gradient of $u$ at $\bx\in \Gamma_2$. For testing purposes, we construct the right-hand side functions $f(\bx)$, $h_1(\bx)$ and $h_2(\bx)$, by selecting $\alpha(\bx)=-e^{-\|\bx\|_2^{2}}$ and taking $u(\bx) = \sin(x_1x_2)$ as the exact solution of \eqref{eq:poissonMixBC}. 
\end{example}

With the same computational  domain and mesh points  as Example \ref{eg:PDE1},  Figure \ref{fig:ex4_5} (left) plots the interior nodes (blue dots), nodes for the Dirichlet  boundary $\Gamma_1$ (black dots) and the Neumann boundary $\Gamma_2$ (red dots) conditions; in particular, 
\begin{align*}
	&{\mathcal X}_0:=\{\bx_j\}_{j=1}^{m_0}\subset \Omega,~m_0=504,\\
	&\what {\mathcal X}_1:=\{\bx_j\}_{j=m_0+1}^{m_0+\what  m_1}\subset \Gamma_1, ~\what m_{1}=30,\\
	&\wtd {\mathcal X}_1:=\{\bx_j\}_{j=m_0+\what m_1+1}^{m_0+\what m_1+\wtd m_1}\subset \Gamma_2, ~\wtd m_1=96
\end{align*}
are the fitting nodes. To treat the  Neumann condition $\nabla u(\bx)\cdot  {\bn}(\bx) = h_2(\bx)$, for any node $\bx_j\in \wtd {\mathcal X}_1$, we have
$$
\nabla u(\bx_j)\cdot  {\bn}(\bx_j) =n_1(\bx_j) \partial_1 u(\bx_j)+ n_2(\bx_j) \partial_2 u(\bx_j)=h_2(\bx_j).
$$
With ${\mathcal X}_1=\what {\mathcal X}_1\cup \wtd {\mathcal X}_1$,  ${\mathcal X}={\mathcal X}_1\cup {\mathcal X}_0$, $m_1=\what m_1+\wtd m_1=126$ and $m=m_0+\what m_1+\wtd m_1=630$,  this leads to
$$
\left(N_1 Q_1^{[1]}(\wtd{\mathcal X}_1)+N_2Q_2^{[1]}(\wtd{\mathcal X}_1)\right)\bc =h_2(\wtd {\mathcal X}_1),
$$
where 
$
N_i = {\diag}(n_i(\bx_{m_0+\what m_1+1}),\dots, n_i(\bx_{m})) \in \bbR^{\wtd m_1\times \wtd m_1}$ for $i=1,2.$

For the discretization matrix ${\bf L}$, it holds that
$
{\bf L}=E_6P_{\Upsilon},
$ where, with $0_m:=0_{m\times m}$,
\begin{align}\nonumber
	P_{\Upsilon}=\diag(\underbrace{I_{m-\wtd m_1},  {0}_{\wtd m_1}}_{{\rm for~} p},\underbrace{{0}_{m-\wtd m_1},N_1}_{{\rm for~} \partial_{1}p},\underbrace{ {0}_{m-\wtd m_1}, N_2}_{{\rm for~} \partial_{2}p}, \underbrace{\Upsilon, {0}_{m-  m_0}}_{{\rm for~} \partial_{1,1}p}, \underbrace{{0}_{m}}_{{\rm for~} \partial_{1,2}p},\underbrace{\Upsilon, {0}_{m-  m_0}}_{{\rm for~} \partial_{2,2}p})\in \bbR^{6m\times 6m},
\end{align}
$\Upsilon=\diag(\alpha(\bx_1),\dots,\alpha(\bx_{m_0}))\in \bbR^{m_0\times m_0}$ and $E_6=[I_m,\dots,I_m]\in \bbR^{m\X 6m}$. Thus  $G=(E_6P_{\Upsilon})^{\T}(E_6P_{\Upsilon})\in \bbR^{6m\X 6m}$. Also, for the LS system \eqref{eq:generalmodel}, we have
\begin{equation}\label{eq:APoisson2} 
	{\bf A}={\bf L}Q^{(2)}= \left[\begin{array}{c} Q^{[0]}({\mathcal X}_0)\\Q^{[0]}(\what{\mathcal X}_1)\\ N_1 Q_1^{[1]}(\wtd{\mathcal X}_1)+N_2Q_2^{[1]}(\wtd{\mathcal X}_1) \end{array}\right]+\left[\begin{array}{c} \Upsilon\left(Q_{1,1}^{[2]}({\mathcal X}_0)+Q_{2,2}^{[2]}({\mathcal X}_0)\right)\\  {0}_{\what m_1\times g}\\  {0}_{\wtd m_1\times g}\end{array}\right] \in \bbR^{m\times g}
\end{equation}
and  
$$
\bb=\left[\begin{array}{c}f({\mathcal X}_0) \\h_1(\what {\mathcal X}_1)\\h_2(\wtd {\mathcal X}_1)\end{array}\right] \in \bbR^{m}.
$$

Applying MV+G-A with these   settings, we plot the numerical solution $p(\bx)\in \bbP_{22}^{2,{\rm tol}}$ ($n=22,g=276$) of $u(\bx)$ and the error $u(\bx)-p(\bx)$ at ${\mathcal X}$, which demonstrates the high accuracy of the approximation solution of $u(\bx)$.
\begin{figure}[h!!!]
	\centering
	\hskip -5mm
	\epsfig{file = 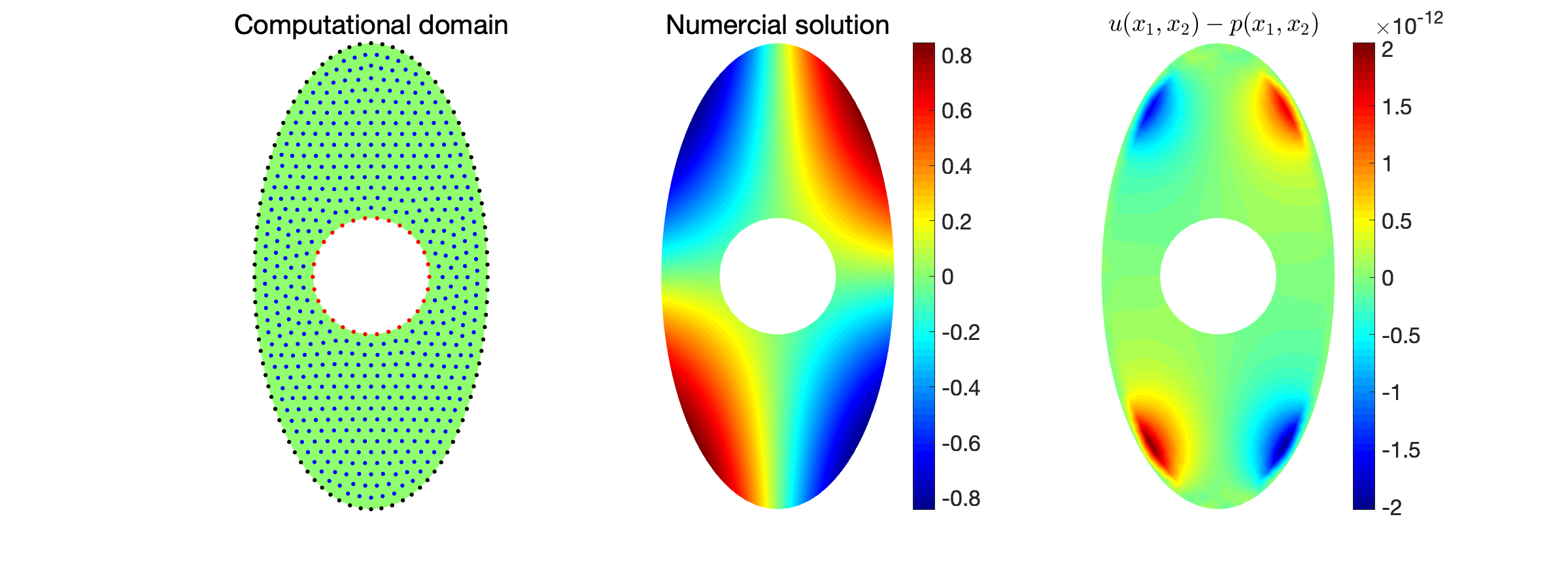,width = 14cm}
	\caption{(left)  fitting nodes;   (middle) numerical solution $p(\bx)$;  (right)  error $u(\bx)-p(\bx)$.} 
	\label{fig:ex4_5}
\end{figure}

\section{Conclusions}\label{sec_conclusion}
In this paper, we have extended the univariate confluent Vandermonde with Arnoldi \cite{nizz:2023} to the multivariate case. Besides the technical treatments of this extension, we also introduced a general symmetric and positive semi-definite $G$-inner product to the Gram-Schmidt process within the Arnoldi process \cite{lirc:2004rept}. The resulting procedure is termed as MV+G-A where the associated symmetric and positive semi-definite matrix $G$ can be specified by a particular application, which yields a well-conditioned (orthonormal) coefficient matrix for the underlying least-squares problem. The approximation multivariate polynomial, represented in a discrete $G$-orthogonal polynomial basis, then can be accurately computed, and the recurrence in the discrete $G$-orthogonal polynomial basis facilitates the evaluation of the polynomial at new nodes. The use of MV+G-A to the multivariate Hermite least-squares problem and Poisson equations with various boundary conditions demonstrates its applicability and flexibility.

\section*{Acknowledgements}
The authors would like to thank Professor Ren-Cang Li at University of Texas at Arlington for mentioning the $G$-Lanczos process and providing the reference \cite{lirc:2004rept} with us. They also thank Chungen Shen for useful discussion. 
\def\noopsort#1{}\def\l{\char32l}\def\v#1{{\accent20 #1}}
  \let\^^_=\v\def\hbk{hardback}\def\pbk{paperback}
\providecommand{\href}[2]{#2}
\providecommand{\arxiv}[1]{\href{http://arxiv.org/abs/#1}{arXiv:#1}}
\providecommand{\url}[1]{\texttt{#1}}
\providecommand{\urlprefix}{URL }

%

\begin{thebibliography}{10}

\bibitem{adca:2020}
\newblock B.~Adcock and J.~M. Cardenas,
\newblock Near-optimal sampling strategies for multivariate function
  approximation on general domains,
\newblock \emph{SIAM J. Maths. of Data Sci.}, \textbf{2} (2020), 607--630,
\newblock \urlprefix\url{https://doi.org/10.1137/19M1279459}.

\bibitem{adcd:2022}
\newblock B.~Adcock, J.~M. Cardenas, N.~Dexter and S.~Moraga,
\newblock \emph{High-Dimensional Optimization and Probability: With a View
  Towards Data Science}, chapter Towards optimal sampling for learning sparse
  approximations in high dimensions, 9--77,
\newblock Springer, 2022.

\bibitem{adhu:2020}
\newblock B.~Adcock and D.~Huybrechs,
\newblock Approximating smooth, multivariate functions on irregular domains,
\newblock \emph{Forum of Mathematics, Sigma}, \textbf{8} (2020), e26.

\bibitem{adsh:2023}
\newblock B.~Adcock and A.~Shadrin,
\newblock Fast and stable approximation of analytic functions from equispaced
  samples via polynomial frames,
\newblock \emph{Constr. Approx.}, \textbf{57} (2023), 257--294,
\newblock \urlprefix\url{https://doi.org/10.1007/s00365-022-09593-2}.

\bibitem{aukl:2021}
\newblock A.~P. Austin, M.~Krishnamoorthy, S.~Leyffer, S.~Mrenna, J.~M\"uller
  and H.~Schulz,
\newblock Practical algorithms for multivariate rational approximation,
\newblock \emph{Comput. Phys. Commun.}, \textbf{261} (2021), 107663.

\bibitem{babl:2002}
\newblock T.~Bagby, L.~Bos and N.~Levenberg,
\newblock Multivariate simultaneous approximation,
\newblock \emph{Constr. Approx.}, \textbf{18} (2002), 569--577.

\bibitem{bach:2009}
\newblock M.~V. Barel and A.~Chesnokov,
\newblock A method to compute recurrence relation coefficients for bivariate
  orthogonal polynomials by unitary matrix transformations,
\newblock \emph{Numer. Algorithms}, \textbf{55} (2009), 383--402.

\bibitem{beck:2000}
\newblock B.~Beckermann,
\newblock {The condition number of real Vandermonde, Krylov and positive
  definite Hankel matrices},
\newblock \emph{Numer. Math.}, \textbf{85} (2000), 553--577.

\bibitem{brnt:2021}
\newblock P.~D. Brubeck, Y.~Nakatsukasa and L.~N. Trefethen,
\newblock Vandermonde with {A}rnoldi,
\newblock \emph{SIAM Rev.}, \textbf{63} (2021), 405--415.

\bibitem{brtr:2022}
\newblock P.~D. Brubeck and L.~N. Trefethen,
\newblock Lightning {S}tokes solver,
\newblock \emph{SIAM J. Sci. Comput.}, \textbf{44} (2022), A1205--A1226.

\bibitem{camv:2005}
\newblock M.~Caliari, S.~D. Marchi and M.~Vianello,
\newblock Bivariate polynomial interpolation on the square at new nodal sets,
\newblock \emph{J. Comput. Appl. Math.}, \textbf{165} (2005), 261--274.

\bibitem{codl:2013}
\newblock A.~Cohen, M.~A. Davenport and D.~Leviatan,
\newblock On the stability and accuracy of least squares approximations,
\newblock \emph{Found. Comput. Math.}, \textbf{13} (2013), 819--834,
\newblock \urlprefix\url{https://doi.org/10.1007/s10208-013-9142-3}.

\bibitem{comi:2017}
\newblock A.~Cohen and G.~Migliorati,
\newblock Optimal weighted least-squares methods,
\newblock \emph{SMAI J. Comput. Math.}, \textbf{3} (2017), 181--203.

\bibitem{defp:2010}
\newblock A.~M. Delgado, L.~Fern{\'a}ndez, T.~E. P{\'e}rez, M.~A. Pi{\~n}ar and
  Y.~Xu,
\newblock Orthogonal polynomials in several variables for measures with mass
  points,
\newblock \emph{Numer. Algorithms}, \textbf{55} (2010), 245--264.

\bibitem{gaut:1963}
\newblock W.~Gautschi,
\newblock On the inverses of {V}andermonde and confluent {V}andermonde
  matrices,
\newblock \emph{Numer. Math.}, \textbf{4} (1962), 117--123.

\bibitem{gaut:2004}
\newblock W.~Gautschi,
\newblock \emph{Orthogonal polynomials: computation and approximation},
\newblock OUP Oxford, 2004.

\bibitem{golr:1983}
\newblock I.~Gohberg, P.~Lancaster and L.~Rodman,
\newblock \emph{Matrices and indefinite scalar products},
\newblock Birkh{\"a}user Verlag, Basel, Boston, Stuttgart, 1983.

\bibitem{high:1990}
\newblock N.~J. Higham,
\newblock Stability analysis of algorithms for solving confluent
  {V}andermonde-like systems,
\newblock \emph{SIAM J. Matrix Anal. Appl.}, \textbf{11} (1990), 23--41.

\bibitem{hoka:2020}
\newblock J.~M. Hokanson,
\newblock Multivariate rational approximation using a stabilized
  {S}anathanan-{K}oerner iteration, 2020,
\newblock \urlprefix\url{arXiv:2009.10803v1}.

\bibitem{kikr:2014}
\newblock I.-P. Kim and A.~R. Kr{\"a}uter,
\newblock {VDR} decomposition of {Chebyshev-Vandermonde} matrices with the
  {A}rnoldi process,
\newblock \emph{Lin. Multilin. Alg.},
\newblock \urlprefix\url{https://doi.org/10.1080/03081087.2024.2335487}.

\bibitem{li:2006a}
\newblock R.-C. Li,
\newblock Asymptotically optimal lower bounds for the condition number of a
  real {Vandermonde} matrix,
\newblock \emph{SIAM J. Matrix Anal. Appl.}, \textbf{28} (2006), 829--844.

\bibitem{li:2005e06}
\newblock R.-C. Li,
\newblock Lower bounds for the condition number of a real confluent
  {Vandermonde} matrix,
\newblock \emph{Math. Comp.}, \textbf{75} (2006), 1987--1995.

\bibitem{li:2008b}
\newblock R.-C. Li,
\newblock {Vandermonde} matrices with {Chebyshev} nodes,
\newblock \emph{Linear Algebra Appl.}, \textbf{428} (2008), 1803--1832.

\bibitem{lirc:2004rept}
\newblock R.-C. Li,
\newblock \emph{Structural Preserving Model Reductions},
\newblock Technical Report 04-02, University of Kentucky, Lexington, 2004.

\bibitem{natr:2021}
\newblock Y.~Nakatsukasa and L.~N. Trefethen,
\newblock Reciprocal-log approximation and planar {PDE} solvers,
\newblock \emph{SIAM J. Numer. Anal.}, \textbf{59} (2021), 2801--2822.

\bibitem{nizz:2023}
\newblock Q.~Niu, H.~Zhang and Y.~Zhou,
\newblock Confluent {Vandermonde with Arnoldi},
\newblock \emph{Appl. Math. Lett.}, \textbf{135} (2023), 108420.

\bibitem{pan:2016}
\newblock V.~Y. Pan,
\newblock How bad are {V}andermonde matrices?,
\newblock \emph{SIAM J. Matrix Anal. Appl.}, \textbf{37} (2016), 676--694.

\bibitem{pest:2004}
\newblock P.-O. Persson and G.~Strang,
\newblock A simple mesh generator in {MATLAB},
\newblock \emph{SIAM Rev.}, \textbf{46} (2004), 329--345.

\bibitem{reic:1993}
\newblock L.~Reichel,
\newblock Construction of polynomials that are orthogonal with respect to a
  discrete bilinear form,
\newblock \emph{Adv. in Comput. Math.}, \textbf{1} (1993), 241--258.

\bibitem{wilk:1988}
\newblock J.~H. Wilkinson,
\newblock \emph{The Algebraic Eigenvalue Problem},
\newblock Monographs on Numerical Analysis, Clarendon Press, Oxford, 1988.

\bibitem{xuyu:2004}
\newblock Y.~Xu,
\newblock On discrete orthogonal polynomials of several variables,
\newblock \emph{Adv. Appl. Math.}, \textbf{33} (2004), 615--632.

\bibitem{yazz:2023}
\newblock L.~Yang, L.-H. Zhang and Y.~Zhang,
\newblock The {L}q-weighted dual programming of the linear {C}hebyshev
  approximation and an interior-point method,
\newblock \emph{Adv. Comput. Math.}, 50:80 (2024).

\bibitem{zhha:2025}
\newblock L.-H. Zhang and S.~Han,
\newblock A convergence analysis of {L}awson's iteration for computing
  polynomial and rational minimax approximations,
\newblock \emph{SIAM J. Numer. Anal.},
\newblock \urlprefix\url{https://arxiv.org/abs/2401.00778v3},
\newblock To appear.

\bibitem{zhsl:2023}
\newblock L.-H. Zhang, Y.~Su and R.-C. Li,
\newblock Accurate polynomial fitting and evaluation via {A}rnoldi,
\newblock \emph{Numerical Algebra, Control and Optimization}, \textbf{14}
  (2024), 526--546.

\bibitem{zhyy:2025}
\newblock L.-H. Zhang, L.~Yang, W.~H. Yang and Y.-N. Zhang,
\newblock A convex dual problem for the rational minimax approximation and
  {L}awson's iteration,
\newblock \emph{Math. Comp.}, \textbf{94} (2025), 2457--2494,
\newblock DOI: https://doi.org/10.1090/mcom/4021.

\bibitem{zhna:2023}
\newblock W.~Zhu and Y.~Nakatsukasa,
\newblock Convergence and near-optimal sampling for multivariate function
  approximations in irregular domains via {V}andermonde with {A}rnoldi, 2023,
\newblock \urlprefix\url{https://arxiv.org/abs/2301.12241}.

\end{thebibliography}
%
%
\end{document}